\documentclass[english,11pt]{article}
\usepackage[english]{babel}
\usepackage{mathptmx}
\usepackage{amssymb,amsmath,enumerate,a4}
\usepackage{colortab}
\usepackage{color}
\usepackage{epsfig}
\usepackage{mathptmx}
\usepackage{tikz}
\usepackage{amsthm}

\usetikzlibrary{calc,intersections,decorations.pathreplacing}

\newtheorem{theorem}{Theorem}{\bfseries}{\itshape}
\newtheorem{lemma}[theorem]{Lemma}{\bfseries}{\itshape}
\newtheorem{corollary}[theorem]{Corollary}{\bfseries}{\itshape}
\newtheorem{claim}{Claim}{\bfseries}{\itshape}
\newtheorem{remark}[theorem]{Remark}{\bfseries}{\itshape}
\newtheorem{observation}[theorem]{Observation}{\bfseries}{\itshape}
\newtheorem{proposition}[theorem]{Proposition}{\bfseries}{\itshape}
\newenvironment{proofclaim}{\noindent { \emph{Proof}.}}  { $\blacksquare$ \smallskip}
{\it}{}
\newtheorem{myclaim}[theorem]{Claim}{\bfseries}{\itshape}

\newcommand{\bb}{\mathbb}
\newcommand{\R}{\bb R}

\newcommand{\Z}{\bb Z}
\newcommand{\N}{\bb N}
\newcommand{\floor}[1]{\left\lfloor#1\right\rfloor}
\newcommand{\ceil}[1]{\left\lceil#1\right\rceil}
\newcommand{\conv}{\operatorname{conv}}

\newcommand{\cone}{\operatorname{cone}}

\renewcommand{\int}{\operatorname{int}}
\newcommand{\dist}{\operatorname{dist}}
\newcommand{\relint}{\operatorname{relint}}
\newcommand{\lin}{\operatorname{lin}}
\newcommand{\aff}{\operatorname{aff}}

\newcommand{\sm}{\setminus}
\newcommand{\relbd}{\operatorname{relbd}}
\newcommand{\rec}{\operatorname{rec}}
\newcommand{\clconv}{\operatorname{\overline{conv}}}
\newcommand{\ang}[1]{\langle #1 \rangle}

\title{Reverse split rank}

\title{Reverse split rank\thanks{This work was supported by the {\em Progetto di Eccellenza 2008--2009} of {\em Fondazione Cassa di Risparmio di Padova e Rovigo}.}}

\author{Michele Conforti\thanks{Dipartimento di Matematica, Universit\`a degli Studi di Padova, Italy.} \and Alberto Del Pia\thanks{Department of Industrial and Systems Engineering $\&$ Wisconsin Institute for Discovery, University of Wisconsin-Madison, United States, {\tt delpia@wisc.edu}.} \\ \and \mbox{Marco Di Summa}\thanks{Dipartimento di Matematica, Universit\`a degli Studi di Padova, Italy.} \and Yuri Faenza\thanks{DISOPT, Institut de math\'ematiques d'analyse et applications, EPFL, Switzerland. }}

\begin{document}

\date{}
\maketitle

\begin{abstract} The {\em reverse split rank} of an integral polyhedron $P$ is defined as the supremum of the split ranks of all rational polyhedra whose integer hull is $P$.
Already in $\R^3$ there exist polyhedra with infinite reverse split rank. We give a geometric characterization of the integral polyhedra in $\R^n$ with infinite reverse split rank.

\smallskip

\end{abstract}

\section{Introduction}\label{sec:intro}

The problem of finding or approximating the convex hull of the integer points in a rational polyhedron is crucial in Integer Programming (see, e.g., \cite{NemWol,sch}). In this paper we consider one of the most well-known procedures used for this purpose: the split inequalities.

Given an integral polyhedron $P\subseteq\R^n$, a {\em relaxation} of $P$ is a rational polyhedron $Q\subseteq\R^n$ such that $P\cap\Z^n=Q\cap\Z^n$, i.e., $\conv(Q\cap\Z^n)=P$, where ``conv'' denotes the convex hull operator.
A {\em split} $S\subseteq\R^n$ is a set of the form $S=\{x\in\R^n:\beta\le ax\le\beta+1\}$ for some integer number $\beta$ and some primitive vector $a\in\Z^n$ (i.e., an integer vector whose entries have greatest common divisor equal to 1). Note that a split does not contain any integer point in its interior $\int S$. Therefore, if $Q$ is a rational polyhedron and $S$ is a split, then the set $\conv(Q\setminus\int S)$ contains the same integer points as $Q$. The {\em split closure} $SC(Q)$ of $Q$ is defined as
\[SC(Q)=\bigcap_{S \mbox{ split}}\conv(Q\setminus\int S).\]
As shown in \cite{CKS}, if $Q$ is a rational polyhedron, its split closure $SC(Q)$ is a rational polyhedron, and it clearly contains the same integer points as $Q$. For $k\in\N$, the {\em $k$-th split closure} of $Q$ is $SC^k(Q)=SC(SC^{k-1}(Q))$, with $SC^0(Q)=Q$. If $Q$ is a rational polyhedron, then there is an integer $k$ such that $SC^k(Q)=\conv(Q\cap\Z^n)$ (see \cite{CKS}); the minimum $k$ for which this happens is called the {\em split rank} of $Q$, and we denote it by $s(Q)$.

While one can verify that the split rank of all rational polyhedra in $\R^2$ is bounded by a constant, there is no bound for the split rank of all rational polyhedra in $\R^3$. Furthermore, even if the set of integer points in $Q$ is fixed, there might be no constant bounding the split rank of $Q$.
%We now give an example of this type taken from \cite{ipco}.
For instance,
let $P\subseteq\R^3$ be the convex hull of the points $(0,0,0)$, $(2,0,0)$ and $(0,2,0)$.
For every $t\ge0$, the polyhedron $Q_t=\conv(P,(1/2,1/2,t))$ is a relaxation of $P$. As shown in~\cite{ipco-journal}, $s(Q_t)\to+\infty$ as $t\to+\infty$.

In this paper we aim at understanding which polyhedra admit relaxations with arbitrarily high split rank. For this purpose,
given an integral polyhedron $P$, we define the {\em reverse split rank} of $P$, denoted $s^*(P)$, as the supremum of the split ranks of all relaxations of $P$:
\[s^*(P)=\sup\{s(Q):\mbox{$Q$ is a relaxation of $P$}\}.\]
For instance, the polyhedron $P$ given in the above example satisfies $s^*(P)=+\infty$.

In order to state our main result, given a subset $K\subseteq\R^n$, we denote by $\int K$ its interior and by $\relint K$ its relative interior. We say that $K$ is (relatively) lattice-free if there are no integer points in its (relative) interior. We denote by $\lin P$ the lineality space of a polyhedron $P$. Furthermore, given to sets $A,B\subseteq\R^n$, $A+B$ denotes the Minkowski sum of $A$ and $B$, defined by $A+B=\{a+b:a\in A,\,b\in B\}$.
(See, e.g., \cite{rockafeller,sch}.)

%Given a nonempty integral polyhedron $P\subseteq\R^n$ and a rational linear subspace $L\subseteq\R^n$, consider the following procedure, which always returns a nonempty face $F$ of $P$ after a finite number of iterations:
%\begin{itemize}
%\item
%set $F:=P$;
%\item
%while there exists a split $S=\{x\in\R^n:\beta\le ax\le\beta+1\}$ such that $\relint(F+L)\subseteq\int S$, redefine $F:=F\cap\{x\in\R^n:ax=\beta+j\}$ for some $j\in\{0,1\}$;
%\item return $F$.
%\end{itemize}
%Any face $F$ of $P$ that can be obtained as output to the above procedure (i.e., for some choice of $S$ and $j$ at every iteration) is called {\em split-exposed} with respect to $L$.
%Note that if there is no split containing $\relint(P+L)$ in its interior, then $P$ itself is split-exposed with respect to $L$.

%We can now state our main result.

\begin{theorem}\label{th:main}
Let $P\subseteq\R^n$ be an integral polyhedron. Then $s^*(P)=+\infty$ if and only if there exist a nonempty face $F$ of $P$ and a rational linear subspace $L\not\subseteq\lin P$ such that
\begin{enumerate}[\upshape(i)]
\item
$\relint(F+L)$ is not contained in the interior of any split,
\item
$G+L$ is relatively lattice-free for every face $G$ of $P$ that contains $F$.
\end{enumerate}
\end{theorem}

Note that for the polyhedron $P$ given in the example above, conditions (i)--(ii) are satisfied by taking $F=P$ and $L$ equal to the line generated by the vector $(0,0,1)$. We also remark that the condition $L\not\subseteq\lin P$ in the theorem implies in particular that $L\ne\{0\}$.
Furthermore, we observe that the dimension of any face $F$ as in the statement of the theorem is at least two.

The analogous concept of {\em reverse Chv\'atal--Gomory (CG) rank} of an integral polyhedron $P$ was introduced in \cite{ipco}.
We recall that an inequality $cx\le\floor{\delta}$ is a \emph{CG inequality} for a polyhedron $Q\subseteq \R^n$ if $c$ is an integer vector and $cx\le\delta$ is valid for $Q$. Alternatively, a CG inequality is a split inequality in which the split $S=\{x\in\R^n:\beta\le ax\le\beta+1\}$ is such that one of the half-spaces $\{x\in\R^n:ax\le\beta\}$ and $\{x\in\R^n:ax\ge\beta+1\}$ does not intersect $Q$.
The \emph{CG closure}, the \emph{CG rank} $r(Q)$, and the {\em reverse CG rank} $r^*(Q)$ of $Q$ are defined as for the split inequalities. The facts that the CG closure of a rational polyhedron is a rational polyhedron and that the CG rank of a rational polyhedron is finite were shown in \cite{sch80}.
In \cite{ipco} the following characterization was proved.

\begin{theorem}[\cite{ipco}]\label{th:CG}
Let $P\subseteq\R^n$ be an integral polyhedron. Then $r^*(P)=+\infty$ if and only if $P\ne\varnothing$ and there exists a one-dimensional rational linear subspace $L\not\subseteq\lin P$ such that
$P+L$ is relatively lattice-free.
\end{theorem}

Note that if an integral polyhedron $P$, a nonempty face $F$ of $P$ and a rational linear subspace $L\not\subseteq\lin P$ satisfy conditions (i)--(ii) of Theorem~\ref{th:main}, then $P\ne\varnothing$ and $P+L$ is relatively lattice-free.
Thus the conditions of Theorem~\ref{th:main} are a strengthening of those of Theorem~\ref{th:CG}. This is not surprising, as every CG inequality is a split inequality, thus $s(Q)\le r(Q)$ for every rational polyhedron $Q$ and $s^*(P)\le r^*(P)$ for every integral polyhedron $P$. Indeed, there are examples of integral polyhedra with finite reverse split rank but infinite reverse CG rank: for instance, the polytope defined as the convex hull of points $(0,0)$ and $(0,1)$ in $\R^2$ (see \cite{ipco}).

The comparison between Theorem~\ref{th:main} and Theorem~\ref{th:CG} suggests that there is some ``gap'' between the CG rank and the split rank. This is not surprising, as the literature already offers results in this direction. For instance, if we consider a rational polytope contained in the cube $[0,1]^n$, it is known that its split rank is at most $n$ \cite{balas}, while its CG rank can be as high as $\Omega(n^2)$ (see \cite{RothSan}; weaker results were previously given in \cite{Eisc,PoSt}). Some more details about the differences between the statements of Theorem~\ref{th:main} and Theorem~\ref{th:CG} will be given at the end of the paper.

We remark that, despite the similarity between the statements of Theorem~\ref{th:main} and Theorem~\ref{th:CG}, the proof of the former result (which we give here) needs more sophisticated tools and is more involved.

The rest of the paper is organized as follows.
In Sect.~\ref{sec:basics} we recall some known facts. In Sect.~\ref{sec:integer-points} we present two results on the position of integer points close to linear or affine subspaces: these results, beside being used in the proof of Theorem~\ref{th:main}, seem to be of their own interest.
The sufficiency of conditions (i)--(ii) of Theorem~\ref{th:main} is proved in Sect.~\ref{sec:suff}, while the necessity of the conditions is shown in Sect.~\ref{sec:nec} for bounded polyhedra, and in Sect.~\ref{sec:nec-unbounded} for unbounded polyhedra.
In Sect.~\ref{sec:mixed} we discuss a connection between the concept of reverse split rank in the pure integer case and that of split rank in the mixed-integer case.
We conclude with some observations in Sect.~\ref{sec:concl}.

\section{Basic facts}\label{sec:basics}

In this section we introduce some notation and present some basic facts that will be used in the proof of Theorem~\ref{th:main}.
We refer the reader to a textbook, e.g.~\cite{sch}, for standard preliminaries that do not appear here.

Given a point $x\in\R^n$ and a number $r>0$, we denote by $B(x,r)$ the closed ball of radius $r$ centered at $x$.
We write $\aff P$ to indicate the affine hull of a polyhedron $P\subseteq\R^n$, $\rec P$ for the recession cone of $P$, and we recall that by $\lin P$ we denote the lineality space of $P$.
The angle between two vectors $v,w\in\R^n$ is denoted by $\phi(v,w)$.
The Euclidean norm of a vector $v\in\R^n$ is denoted by $\|v\|$, while $\dist(A,B):=\inf\{\|a-b\|:a \in A,\,b\in B\}$ is the Euclidean distance between two subsets $A,B\subseteq\R^n$. (If $A=\{a\}$, we write $\dist(a,B)$ instead of $\dist(\{a\},B)$.)
Given subsets $S_1,\dots,S_k$ of $\R^n$, we indicate with $\langle S_1,\dots,S_k\rangle$ the linear subspace of $\R^n$ generated by the union of $S_1,\dots,S_k$. (If $S$ is a subset of $\R^n$ and $v\in\R^n$, we write $\langle S,v\rangle$ instead of $\langle S,\{v\}\rangle$ and $\langle v\rangle$ instead of $\langle\{v\}\rangle$.)
Further, $L^\bot$ is the orthogonal complement of a linear subspace $L\subseteq\R^n$.
Finally, we denote by $\cone(v_1,\dots,v_k)$ the set of conic combinations of vectors $v_1,\dots,v_k\in\R^n$.

\subsection{Unimodular transformations}\label{sec:unimod}

A \emph{unimodular transformation} $u: \R^n \rightarrow \R^n$ maps a point $x \in \R^n$ to $u(x)=Ux + v$, where $U$ is an $n\times n$ unimodular matrix (i.e., a square integer matrix with $\mathopen|\det(U)|=1$) and $v\in \Z^n$. It is well-known (see e.g.~\cite{sch}) that $U$ is a unimodular matrix if and only if so is $U^{-1}$. Furthermore, a unimodular transformation is a bijection of both $\R^n$ and $\Z^n$. It follows that if $Q\subseteq \R^n$ is a rational polyhedron and $u : \R^n \rightarrow \R^n$ is a unimodular transformation, then the split rank of $Q$ coincides with the split rank of $u(Q)$.

The following basic fact will prove useful: if $L\subseteq\R^n$ is a rational linear subspace of dimension $d$, then there exists a unimodular transformation that maps $L$ to the subspace $\{x\in\R^n:x_{d+1}=\dots=x_n=0\}$; in other words, $L$ is equivalent to $\R^d$ up to a unimodular transformation.

\subsection{Some properties of CG and split rank}

We will use the following result (see \cite[Lemma~10]{non-full-dim}) and its easy corollary.

\begin{lemma}\label{lem:ub-chv}
For every $n\in\N$ there exists a number $\theta(n)$ such that the following holds:
for every rational polyhedron $Q \subseteq \R^n$, $c \in \Z^n$ and $\delta,\delta'\in\R$ with $\delta'\ge\delta$, where $cx \le
\delta$ is valid for $\conv(Q\cap\Z^n)$ and $cx \le \delta'$ is valid for $Q$, the inequality $cx
\le \delta$ is valid for the $p$-th CG closure of $Q$, where $p=(\lfloor \delta'\rfloor - \lfloor\delta
\rfloor) \theta(n)+1$.
\end{lemma}

\begin{corollary}\label{cor:ub-chv}
Given an integral polytope $P\subseteq\R^n$ and a bounded set $B$ containing $P$, there exists an integer $N$ such that $r(Q)\le N$ for all relaxations $Q$ of $P$ contained in $B$.
\end{corollary}

We also need the following lemma.

%\begin{lemma}[\cite{CKS}]\label{lem:face}
%If $F$ is a face of an integral polyhedron $P$, then $SC^k(F)=SC^k(P)\cap F$ for every $k\in\N$.
%\end{lemma}

\begin{lemma}\label{lem:split-faces}
Let $Q\subseteq\R^n$ be a rational polyhedron contained in a split $S$, where $S=\{x\in\R^n:\beta\le ax\le\beta+1\}$. Let $Q^0$ (resp., $Q^1$) be the face of $Q$ induced by the inequality $ax\ge\beta$ (resp., $ax\le\beta+1$). Then $s(Q)\le\max\{s(Q^0),s(Q^1)\}+1$.
\end{lemma}

\begin{proof}
For $j=0,1$, since $Q^j$ is a (possibly non-proper) face of $Q$, we have $SC(Q^j)=SC(Q)\cap Q^j$ (see \cite{CKS}).
Then, for $k=\max\{s(Q^0),s(Q^1)\}$, both $SC^k(Q)\cap Q^0$ and $SC^k(Q)\cap Q^1$ are integral polyhedra. It follows that after another application of the split closure (actually, the split $S$ is sufficient) we obtain an integral polyhedron.
\end{proof}

%\begin{proof}
%For $i=0,1$, define the hyperplane $H^i=\{x\in\R^n:ax=\beta+i\}$. Since $Q^i$ is a face of $Q$, by Lemma~\ref{lem:face} we have $SC(Q^i)=SC(Q)\cap H^i$ for $i=0,1$. Then, for $k=\max\{s(Q^0),s(Q^1)\}$, both $SC^k(Q)\cap H^0$ and $SC^k(Q)\cap H^1$ are integral polyhedra. It follows that after another application of the split closure (actually, the split $S$ is sufficient) we obtain an integral polyhedron.
%\end{proof}

\subsection{Maximal lattice-free convex sets}\label{sec:lf}

A {\em maximal lattice-free} convex set is a convex set that is not strictly contained in any lattice-free convex set.
A result in~\cite[Theorem 2]{BasuDirichelet} (see also~\cite{Lov}) states that a maximal lattice-free convex set in $\R^n$ is either an irrational hyperplane or a polyhedron $P+L$, where $P$ is a polytope and $L$ is a rational linear subspace.
In particular, since a full-dimensional set is never contained in a hyperplane, every full-dimensional lattice-free convex set is contained in a set of the form $P+L$, where $P$ is a polytope and $L$ is a rational linear subspace.

\subsection{Lattice width}

The lattice width $w(K)$ of a closed convex set $K\subseteq\R^n$ (with respect to the integer lattice $\Z^n$) is defined by
\[w(K) := \inf_{c \in \Z^n \sm \{0\}}\left\{\sup_{x \in K} cx - \inf_{x \in K} cx \right\}.\]
If $K$ is full-dimensional and $w(K)<+\infty$, then there exists a nonzero integer vector $c$ for which
\[w(K) = \max_{x \in K} cx - \min_{x \in K} cx .\]
Furthermore, $c$ is a primitive vector. (See, e.g., \cite{Bar}.)

We will use the following extension of the well-known Flatness Theorem of Khintchine~\cite{Khi} (see also \cite[Chapter~7]{Bar}), which is taken from \cite[Corollary~5]{non-full-dim} (see also \cite[Theorem~(4.1)]{KanLov}).
\begin{lemma}\label{lem:flat}
For every $k\in\N$ and every convex body $K \subseteq \R^n$ with $|K \cap k\Z^n|=1$, one has $w(K)\leq \omega(n,k)$, where $\omega$ is a function depending on $n$ and $k$ only.
\end{lemma}

\subsection{Compactness}

The proof of Theorem~\ref{th:main} exploits the notion of compactness and sequential compactness, which we recall here.
A subset $K$ of a topological space is {\em compact} if every collection of open sets covering $K$ contains a finite subcollection which still covers $K$. It is well-known that a subset of $\R^n$ is compact (with respect to the usual topology of $\R^n$) if and only if it is closed and bounded.
For a normed space (such as $\R^n$) the notion of compactness coincides with that of sequential compactness: a set $K$ is {\em sequentially compact} if every sequence $(x_i)_{i\in\N}$ of elements of $K$ admits a subsequence that converges to an element of $K$.

\section{On integer points close to subspaces}\label{sec:integer-points}

A result given in \cite{BasuDirichelet}, based on Dirichlet's approximation theorem (see, e.g., \cite{sch}), shows that for each line passing through the origin there are integer points arbitrarily close to the line and arbitrarily far from the origin. (Note that if the line is not rational, then the origin is the only integer point lying on in it.) We give here a strengthening of that result, showing that for every line passing through the origin the integer points that are ``very close'' to the line are not too far from each other. Furthermore, this result is presented in a more general version, valid for every linear subspace. This lemma will be used in the proof of Theorem~\ref{th:main}, but we find it interesting in its own right.

\begin{lemma}\label{lem:integer-point}
Let $L\subseteq\R^n$ be a linear subspace and fix $\delta>0$. Then there exists $R>0$ such that, for every $x\in L$, there is an integer point $y$ satisfying $\|y-x\|\le R$ and $\dist(y,L)\le\delta$.
\end{lemma}

\begin{proof}
The proof is a double induction on $n$ and $d:=\dim L$.
The statement is easily verified for $n=1$; so we fix $n\ge2$ and assume by induction that the result holds in dimension smaller than $n$.
The proof is now by induction on $d$.

\paragraph{\sc Base step}
The statement is trivial if $d=0$. We show that it is correct for $d=1$, as well. So we now assume that $L=\langle v\rangle$ for some $v\in\R^n\setminus\{0\}$.

\begin{claim}
If there is no row-vector $a\in\Z^n\setminus\{0\}$ such that $av=0$, then the result of the lemma holds.
\end{claim}

\begin{proofclaim}
Define the $(n-1)$-dimensional ball $B'=B(0,\delta)\cap\langle v\rangle^\bot$. Note that $B(0,\delta)+\langle v\rangle=B'+\langle v\rangle$.
Let us denote by $\relbd B'$ the relative boundary of $B'$.
For every $z\in\relbd B'$, denote by $C(z)$ the intersection of $B'$ with the open cone of revolution of direction $z$ and angle $\pi/6$.
We claim that for every $z\in\relbd B'$ and $z',w\in C(z)$, $z'-w\in B'$ (see Fig.~\ref{fig:ball}).
To see this, assume wlog that $\|z'\|\le\|w\|\le\delta$. Then
\[\begin{split}
\|z'-w\|^2 & = \|z'\|^2+\|w\|^2-2\,z'\cdot w\\
&=\|z'\|^2+\|w\|^2-\|z'\|\|w'\|\\
&\le\|z'\|^2+\delta\|w\|-\|z'\|\|w'\|\\
&=\|z'\|(\|z'\|-\|w\|)+\delta(\|w\|-\delta)+\delta^2\le\delta^2,
\end{split}\]
where the second equality holds because the angle between $z'$ and $w$ is at most $\pi/3$ and thus $\frac{z'\cdot w}{\|z'\|\|w'\|}=\cos(\pi/3)=1/2$.
This proves that $z'-w\in B'$ whenever $z',w\in C(z)$.

Note that $B'$ is a compact set, $B'=\bigcup_{z\in\relbd B'}C(z)$, and each $C(z)$ is an open set in $B'$. Then there exist $z_1,\dots,z_m\in\relbd B'$ such that $B'=\bigcup_{i=1}^mC(z_i)$. Define $C_i=C(z_i)$ for $i=1,\dots, m$.

\begin{figure}
\begin{center}
\psset{unit=7mm}
\psset{linewidth=.5pt}
\begin{pspicture}(6,6.4)
\pscircle(3,3){3}
\psline(2,5.8)(3,3)(4,5.8)
\put(3,3){\circle*{0.15}}
\put(2.7,2.7){0}
\put(2.3,5.5){\circle*{0.15}}
\put(2.4,5.5){$w$}
\put(3.1,3.6){\circle*{0.15}}
\put(3,3.8){$z'$}
\psline[linestyle=dashed](3.1,3.6)(3.8,1.1)
\put(3.8,1.1){\circle*{0.15}}
\put(4,1){$z'-w$}
\put(2.6,4.7){$C(z)$}
\put(1.3,2.7){$B'$}
\put(3,6.15){$z$}
\put(3,6){\circle*{0.15}}
\end{pspicture}
\caption{Illustration of the proof of the claim.}
\label{fig:ball}
\end{center}
\end{figure}
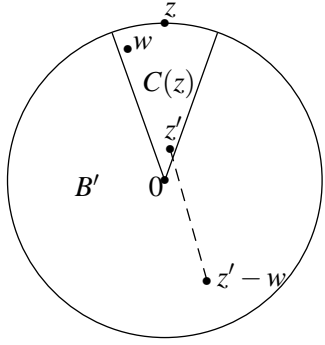

We claim that for every $i=1,\dots, m$, the set $C_i+\langle v\rangle$ is not lattice-free. To see this, assume by contradiction that $C_i+\langle v\rangle$ is lattice-free. Then it is contained in a full-dimensional maximal lattice-free polyhedron, i.e., a set of the form $P+L'$ for some polytope $P\subseteq \R^n$ and some rational linear subspace $L'\subseteq\R^n$ (recall Sect.~\ref{sec:lf}). Note that $v\in L'$. However, this is not possible, as we assumed that there is no row-vector $a\in\Z^n\setminus\{0\}$ such that $av=0$ (i.e., $v$ is not contained in any rational subspace). Therefore $C_i+\langle v\rangle$ is not lattice-free for $i=1,\dots,m$. This implies that $C_i+\cone(v)$ is not lattice-free for $i=1,\dots,m$.

For $i=1,\dots,m$, let $w_i$ be an integer point in $C_i+\cone(v)$; note that $w_1,\dots,w_m\notin\langle v\rangle^\bot$ because of the hypothesis of the claim. Observe that for every $z\in B'+\langle v\rangle$, at least one of the points $z-w_1,\dots,z-w_m$ is still in $B'+\langle v\rangle$ (just choose $i$ such that $z\in C_i+\langle v\rangle$).
We define
\[M=\max_{i=1,\dots,m}\dist(w_i,B')>0, \quad \mu=\min_{i=1,\dots,m}\dist(w_i,B')>0,\]
and show that the statement of the lemma holds by choosing $R=M+\delta$.

Take any $x\in\langle v\rangle$. Since $v$ is contained in no rational subspace, the set $x+B'+\cone(v)$ contains an integer point $z$ (see, e.g., \cite[Lemma 2.2]{BasuDirichelet}).
If $\dist(z,x+B')>M$, we choose $i\in\{1,\dots,m\}$ such that the point $z'=z-w_i$ is still in $B(x,\delta)+\cone(v)$. Since  $0<\dist(z',x+B')\le\dist(z,x+B')-\mu$, and $\mu>0$, by iterating this procedure a finite number of times we arrive at an integer point $y\in B(x,\delta)+\cone(v)$ such that $\dist(y,x+B')\le M$.
Then $\|y-x\|\le M+\delta=R$ and $\dist(y,\langle v\rangle)\le\delta$. This concludes the proof of the claim. 
\end{proofclaim}

We can now prove the lemma for $d=1$. Let $S\subseteq\R^n$ be a minimal rational subspace containing $v$. If $S=\R^n$, then the hypothesis of the claim is satisfied and we are done. So assume that $\dim S<n$. In this case, by applying a unimodular transformation we can reduce ourselves to the case in which $S=\R^{\dim S}\times\{0\}^{n-\dim S}$ and use induction, as $S$ is now equivalent to an ambient space of dimension smaller than $n$.
Though unimodular transformations do not preserve distances, there exist positive constants $c_1\le c_2$ (depending only on the transformation) such that the distance between any two points (or sets) is scaled by a factor between $c_1$ and $c_2$, so the arguments can be easily adapted.

\paragraph{\sc Inductive step}
Fix $d\ge 2$ and assume that the lemma holds for every subspace of $\R^n$ of dimension smaller than $d$.
Fix any $v\in L$ and define $L'=L\cap\langle v\rangle^\bot$.
By the base step of the induction, there exists $R_1>0$ such that, for every $x\in\langle v\rangle$, there is an integer point $y$ satisfying $\|y-x\|\le R_1$ and $\dist(y,\langle v\rangle)\le\delta/2$.
Furthermore, by induction, there exists $R_2>0$ such that, for every $x\in L'$, there is an integer point $y$ satisfying $\|y-x\|\le R_2$ and $\dist(y,L')\le\delta/2$. Note that this remains true also if we replace $L'$ with an affine subspace $L_1$ obtained by translating $L'$ by an integer vector. We show that the result of the lemma holds with $R=R_1+R_2$; see Fig.~\ref{fig:int-pt-ind} to follow the proof.

Take any $x\in L$ and decompose it by writing $x=x_1+x_2$, where $x_1\in\langle v\rangle$ and $x_2\in L'$.
Let $y_1$ be an integer point satisfying $\|y_1-x_1\|\le R_1$ and $\dist(y_1,\langle v\rangle)\le\delta/2$. Define $L_1=y_1+L'$.
Let $x'=y_1+x_2$; note that $\|x'-x\|=\|y_1-x_1\|\le R_1$.
Since $L_1$ is a translation of $L'$ by an integer vector, there is an integer point $y_2$ satisfying $\|y_2-x'\|\le R_2$ and $\dist(y_2,L_1)\le\delta/2$. Now, $\dist(y_2,L)\le\dist(y_2,L_1)+\dist(L_1,L)\le\delta/2+\delta/2\le\delta$. Furthermore, $\|y_2-x\|\le\|y_2-x'\|+\|x'-x\|\le R_2+R_1=R$.
This concludes the proof of Lemma~\ref{lem:integer-point}.
\end{proof}

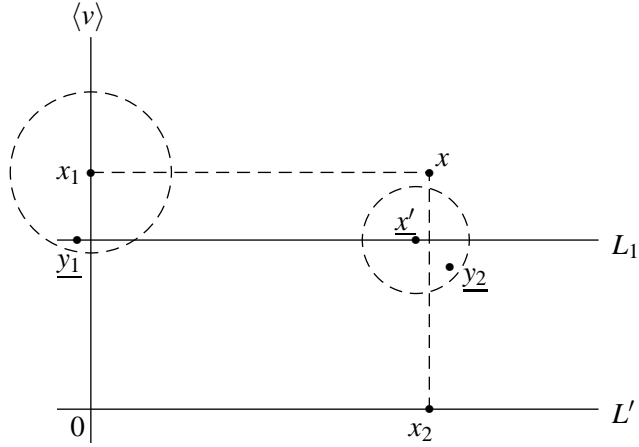
\begin{figure}
\begin{center}
\psset{unit=9mm}
\psset{linewidth=.5pt}
\begin{pspicture}(10,6.6)
\psline(1,0.5)(9,0.5)
\psline(1.5,0)(1.5,6)
\put(1.2,0.1){0}
\put(9.2,0.3){$L'$}
\put(1.2,6.2){$\langle v\rangle$}
\put(6.5,4){\circle*{.1}}
\put(6.6,4.1){$x$}
\psline[linestyle=dashed](1.5,4)(6.5,4)(6.5,0.5)
\put(1.5,4){\circle*{.1}}
\put(1,3.9){$x_1$}
\put(6.5,0.5){\circle*{.1}}
\put(6.2,0.1){$x_2$}
\pscircle[linestyle=dashed](1.5,4){1.2}
\put(1.3,3){\circle*{.1}}
\put(1,2.6){$\underline{y_1}$}
\psline(1,3)(9,3)
\put(9.2,2.8){$L_1$}
\put(6.3,3){\circle*{.1}}
\put(6,3.15){$\underline{x'}$}
\pscircle[linestyle=dashed](6.3,3){0.8}
\put(6.8,2.6){\circle*{.1}}
\put(7,2.4){$\underline{y_2}$}
\end{pspicture}
\caption{Illustration of the inductive step in the proof of Lemma~\ref{lem:integer-point}. The space $L$ is represented. Underlined symbols indicate points that do not necessarily belong to $L$; in other words, their orthogonal projection onto $L$ is represented. The circle on the left has radius $R_1$, the one on the right has radius $R_2$.}
\label{fig:int-pt-ind}
\end{center}
\end{figure}

We now prove a result that gives sufficient conditions guaranteeing that a non-full-dimensional simplex of a special type is ``very close'' to an integer point.

\begin{lemma}\label{lem:very-far-things}
For a given $k\in\N$, let $L_0\subseteq L_1\subseteq \dots \subseteq L_{k}\subseteq H$ be a sequence of linear subspaces, where $\dim L_t=t$ for $t=0,\dots,k$, and $H$ is a rational subspace. Then for every $\delta>0$ there exists $M>0$ such that the following holds: for every $x\in L_{k}$ and for every $y_1,\dots,y_{k}$ satisfying $y_t\in x+L_t$ and $\dist(y_t,x+L_{t-1})\ge M$ for $t=1,\dots,k$, one has
\[\dist(\conv(x,y_1,\dots,y_{k}), H\cap\Z^n)\leq \delta.\]
\end{lemma}

\begin{proof}
We first observe that it is enough to show the result for $H=\R^n$. Indeed, if $H$ is a $d$-dimensional rational subspace of $\R^n$ with $d<n$, we can apply a unimodular transformation mapping $H$ to $\R^d\times\{0\}^{n-d}$ to reduce ourselves to the case in which $H$ coincides with the ambient space.
As in the proof of Lemma~\ref{lem:integer-point}, there exist positive constants $c_1\le c_2$ (depending only on the transformation) such that the distance between any two points is scaled by a factor between $c_1$ and $c_2$, and this is enough to conclude.

Therefore in the following we assume $H=\R^n$.
We proceed by induction on $k$.

\paragraph{\sc Base step}
The case $k=0$ is trivial. Here we assume $k=1$. Apply Lemma~\ref{lem:integer-point} with $L=L_1$ and $\delta=\delta$, and define $M=2R$. Pick $y \in x+L_1$ at distance at least $M$ from $x$. Because of Lemma~\ref{lem:integer-point}, we know that there exists an integer point $z$ at distance at most $\delta$ from $L_1$ and at most $M/2$ from the middle point of $x$ and $y$. The latter condition implies that the orthogonal projection of $z$ onto $L_1$ lies between $x$ and $y$, hence $z$ is at distance at most $\delta$ from $\conv(x,y)$.

\paragraph{\sc Inductive step}
We prove the lemma when subspaces $L_0\subseteq L_1 \subseteq \dots \subseteq L_k$ are given ($k\ge2$), assuming that the result holds for shorter sequences of subspaces.

By Lemma~\ref{lem:integer-point}, there exists $R>0$ such that, for every $z\in L_k$, there is an integer point $u$ satisfying $\|u-z\|\le R$ and $\dist(u,L_k)\le\delta/2$.

By induction, there exists $M'>0$ such that if $x'\in L_{k-1}$ and $y_1,\dots,y_{k-1}$ satisfy $y_t\in x'+L_t$ and $\dist(y_t,x'+L_{t-1})\ge M'$ for $t=1,\dots,k-1$, then $\dist(\conv(x',y_1,\dots,y_{k-1}), \Z^n)\leq \delta/2$.

We show that the result holds if we take $M=\max\{2M',4R\}$ (see Fig.~\ref{fig:very-far-things} to follow the proof).
So fix $x\in L_k$ and $y_1,\dots,y_k$ satisfying $y_t\in x+L_t$ and $\dist(y_t,x+L_{t-1})\ge M$ for $t=1,\dots,k$.
Let $v$ be the unit-norm vector in $L_k\cap L_{k-1}^\bot$ such that $y_k \in x+ L_{k-1}+ \alpha v$ for some $\alpha \geq 0$ (in words: with respect to $x+L_{k-1}$, the point $y_k$ lies on the side pointed by $v$). If we define $z=x+Rv\in L_k$, then there exists an integer point $u$ such that $\|u-z\|\le R$ and $\dist(u,L_k)\le\delta/2$. Let $\tilde u$ be the orthogonal projection of $u$ onto $L_k$. Note that $\dist(\tilde u,x+L_{k-1})\le 2R$.

Define $\tilde x$ as the unique point in $[x,y_k]\cap(\tilde u+L_{k-1})$; for $t=1,\dots,k-1$, define $\tilde y_t$ as the unique point in $[y_t,y_k]\cap(\tilde u+L_{k-1})$. Since
\[\dist(y_k,x+L_{k-1})\ge M\ge 4R\ge2\dist(\tilde u,x+L_{k-1})\]
and since $y_t\in x+L_{k-1}$ for $t=1,\dots,k-1$, we have, for $t=1,\dots,k-1$,
\[\dist(\tilde y_t,\tilde x+L_{t-1})\ge\frac12\dist(y_t,x+L_{t-1})\ge \frac M2\ge M'.\]

Now define $x',y'_1,\dots,y'_{k-1}$ as the points obtained by projecting $\tilde x,\tilde y_1,\dots,\tilde y_{k-1}$ orthogonally onto $u+L_{k-1}$, which is a translation of $L_{k-1}$ by an integer vector. Note that the points $x',y'_1,\dots,y'_{k-1}$ are obtained by translating $\tilde x,\tilde y_1,\dots,\tilde y_{k-1}$ by the vector $u-\tilde u$, whose norm is at most $\delta/2$. Since we still have
\[\dist(y'_t,x'+L_{t-1})\ge M', \quad t=1,\dots,k-1,\]
by induction we obtain an integer point $\bar z$ such that $\dist(\conv(x',y'_1,\dots,y'_{k-1}),\bar z)\le\delta/2$. Then
\begin{align*}
\dist(\conv(x,y_1\,\dots,y_k),\bar z) & \leq  \dist(\conv(\tilde x,\tilde y_1\,\dots,\tilde y_{k-1}),\bar z) \\
& \leq  \dist(\conv(x',y_1'\,\dots,y_{k-1}'),\bar z) + \delta/2 \\
& \leq  \delta,
\end{align*}
where the first inequality holds because $\conv(\tilde x,\tilde y_1\,\dots,\tilde y_{k-1})\subseteq\conv(x,y_1\,\dots,y_k)$
This concludes the proof.
\end{proof}

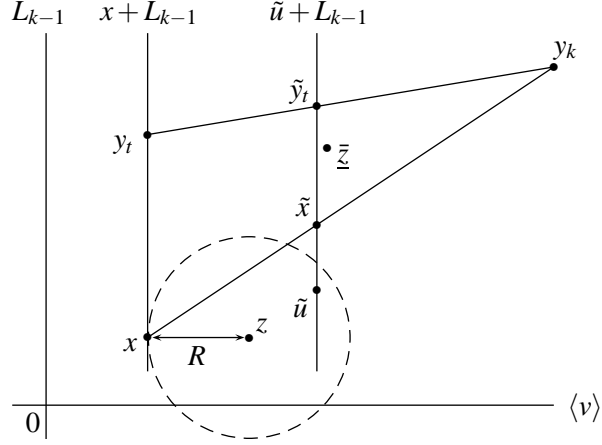
\begin{figure}
\begin{center}
\psset{unit=9mm}
\psset{linewidth=.5pt}
\begin{pspicture}(8.7,6.6)
\psline(0,0.5)(8,0.5)
\psline(0.5,0)(0.5,6)
\put(0.2,0.1){0}
\put(8.2,0.35){$\langle v\rangle$}
\put(0,6.2){$L_{k-1}$}
\put(2,1.5){\circle*{.1}}
\put(1.65,1.3){$x$}
\psline(2,1)(2,6)
\put(1.3,6.2){$x+L_{k-1}$}
\put(3.5,1.5){\circle*{.1}}
\put(3.6,1.6){$z$}
\psline{<->}(2.07,1.5)(3.43,1.5)
\put(2.6,1.1){$R$}
\pscircle[linestyle=dashed](3.5,1.5){1.5}
\put(4.5,2.2){\circle*{.1}}
\put(4.15,1.8){$\tilde u$}
\psline(4.5,1)(4.5,6)
\put(3.8,6.2){$\tilde u+L_{k-1}$}
\put(2,4.5){\circle*{.1}}
\put(1.5,4.3){$y_t$}
\put(8,5.5){\circle*{.1}}
\put(8,5.7){$y_k$}
\psline(2,4.5)(8,5.5)(2,1.5)
\put(4.5,4.9167){\circle*{.1}}
\put(4.1,5.05){$\tilde y_t$}
\put(4.5,3.1667){\circle*{.1}}
\put(4.2,3.3){$\tilde x$}
\put(4.65,4.3){\circle*{.1}}
\put(4.8,4.1){$\underline{\bar z}$}
\end{pspicture}
\caption{Illustration of the inductive step in the proof of Lemma~\ref{lem:very-far-things}. The space $L_k$ is represented. Symbol $\bar z$ is underlined to indicate that $\bar z$ does not necessarily belong to $L_k$; in other words, its orthogonal projection onto $L_k$ is represented. Points $x'$, $y'_t$ and $u'$ are not depicted; however, they project down to $\tilde x$, $\tilde y_t$ and $\tilde u$ respectively, and their distance from the corresponding projected point is at most $\delta/2$.}
\label{fig:very-far-things}
\end{center}
\end{figure}

\section{Proof of sufficiency}\label{sec:suff}

In this section we prove that if $F$ and $L$ satisfying conditions (i)--(ii) of Theorem~\ref{th:main} exist, then $P$ has infinite reverse split rank.

By hypothesis, $F$ and $P$ are nonempty.
Since $L$ is a rational subspace, it admits a basis $v^1,\dots,v^k \in \Z^n$.
Fix $\bar x \in \relint F$, and for $\lambda \ge 0$ define the polyhedra
\[Q_F^\lambda=\clconv(F, \bar x \pm \lambda v^1, \dots, \bar x \pm \lambda v^k),\quad
Q_P^\lambda=\clconv(P, \bar x \pm \lambda v^1, \dots, \bar x \pm \lambda v^k),\]
where $\clconv$ denotes the closed convex hull.
Clearly $Q_F^\lambda \subseteq Q_P^\lambda$ for every $\lambda \ge 0$.
As $\bar x \in \relint F$ and $F+L$ is relatively lattice-free, it follows that $Q_F^\lambda$ is a relaxation of $F$ for every $\lambda \ge 0$.
We now show that also $Q_P^\lambda$ is a relaxation of $P$ for every $\lambda \ge 0$.

%\begin{lemma}\label{lem:G}
%Let $P\subseteq\R^n$ be an integral polyhedron. If $F$ and $L$ satisfy the conditions of Theorem~\ref{th:main}, then $G+L$ is relatively lattice-free for every face $G$ of $P$ that contains $F$.
%\end{lemma}

%\begin{proof}
%Since $F+L$ is relatively lattice-free by (ii), we can assume that $F\subsetneq G$.
%Recall that $F$ is split-exposed with respect to $L$ by (i).
%This means that the procedure given in Section~\ref{sec:intro} returns $F$ for a suitable choice of a split $S$ and an index $j$ at each iteration. Let $S_1,\dots,S_m$ be such a sequence of splits. Note that $m\ge1$, as $F$ is a proper face of $P$ (because $F\subsetneq G$). For every $i=1,\dots,m$, we denote by $H_i$ the hyperplane bounding $S_i$ and containing $F$ (in other words, $H_i$ is the hyperplane selected at the $i$-th iteration of the procedure).
%Similarly, we denote by $H'_i$ the other hyperplane bounding $S_i$, for $i=1,\dots,m$.
%Note that $F=P\cap H_1\cap\dots\cap H_m$.
%Let $\ell$ be the maximum index such that $G\subseteq P\cap H_1\cap\dots\cap H_{\ell-1}$ but $G\not\subseteq P\cap H_1\cap\dots\cap H_{\ell}$ ($\ell$ is well-defined because $F\subsetneq G$).
%Since $\relint(P\cap H_1\cap\dots\cap H_{\ell-1})\subseteq\int(S_\ell)$, and $G$ contains points of both $H_\ell$ and $H_\ell'$, we deduce that
%$\relint G\subseteq \int S_\ell$. Since $L\subseteq \lin S_\ell$, this implies that $\relint(G+L)\subseteq\int S_\ell$ and thus $G+L$ is relatively lattice-free.
%
%\end{proof}

\begin{myclaim}\label{claim:relax}
$Q_P^\lambda$ is a relaxation of $P$ for every $\lambda \ge 0$.
\end{myclaim}

\begin{proof}
Fix $\lambda\ge0$ and assume by contradiction that $Q_P^\lambda$ contains an integer point $z\notin P$. Since $Q_P^\lambda\subseteq P+L$, $z\in P+L$. Let $G'$ be a minimal face of $P+L$ containing $z$; thus $z\in\relint G'$. If $ax\le\beta$ is an inequality that defines face $G'$ of $P+L$, then $ax\le\beta$ defines a face $G$ of $P$ such that $G'=G+L$. Then $z$ is an integer point in $\relint (G +L)$, thus $G+L$ is not relatively lattice-free. By condition~(ii) of Theorem~\ref{th:main}, this implies that $G$ is a face of $P$ not containing $F$, and thus $\bar x\notin G$. Recall that $z\in Q_P^\lambda=\clconv(P, \bar x \pm \lambda v^1, \dots, \bar x \pm \lambda v^k)$. Note that all points in $P\cup\{\bar x \pm \lambda v^1, \dots, \bar x \pm \lambda v^k\}$ satisfy $ax\le\beta$, and every point of the form $\bar x \pm \lambda v^i,i=1,\dots,k$ satisfies $ax<\beta$, as $\bar x\notin G$. Since $az=\beta$, it follows that $z$ is a convex combination of points in $P$. Then $z\in P$, a contradiction.

\end{proof}

Let $r>0$ be the radius of the largest ball in $\aff F$ centered at $\bar x$ and contained in $F$.
Clearly $r$ is finite, because otherwise $F = \aff F$ and so $F+L$ would be an integral affine subspace of $\R^n$, thus not relatively lattice-free, contradicting condition~(ii) of Theorem~\ref{th:main}.

Since $F$ is an integral polyhedron, it can be written in the form
\begin{equation} \label{eq: dec}
F = \conv\{g^1,\dots,g^p\} + \cone \{h^1,\dots,h^q\},
\end{equation}
where $g^i$, $i=1,\dots,p$, are integer points (one in each minimal face of $F$), and $h^i$, $i=1,\dots,q$, are integer vectors; here $p\ge1$ and $q\ge0$.
Let
$$R^g = \max \{\|\bar x - g^i \|: i=1,\dots,p \}, \quad R^h = \max \{0,2\|h^i \|: i=1,\dots,q \}.$$
(Note that the 0 in the latter definition makes $R^h$ well defined.)
Both $R^g$ and $R^h$ are finite, and so is $R = \max\{R^g,R^h\}$.

We will show below that, for each $\lambda\ge1$, $SC(Q_F^\lambda)$ contains the two points
\begin{equation}\label{eq:2pts}
\bar x \pm \min\left\{(\lambda-1),\frac{r}{2(r+R)} \lambda\right\}  v^i
\end{equation}
for every $i=1,\dots,k$.
As $Q_F^\lambda \subseteq Q_P^\lambda$, we have $SC(Q_F^\lambda) \subseteq SC(Q_P^\lambda)$.
As $\lambda$ was chosen arbitrarily, and at least one of the $2k$ points in \eqref{eq:2pts} is not in $P$ for $\lambda$ large enough (because $L$ is not contained in $\lin P$), this implies that $P\subsetneq SC(Q_P^\lambda)$, i.e., $s(Q_P^\lambda)>1$. If $\lambda$ is large then the argument can be iterated, showing that $s(Q_P^\lambda)\to +\infty$ as $\lambda\to +\infty$, hence $s^*(P)=+\infty$.\smallskip

It remains to prove that $SC(Q_F^\lambda)$ contains the two points given in \eqref{eq:2pts} for every $i=1,\dots,k$.
To do so, we prove that for every split $S$, the set $\conv(Q_F^\lambda \sm \int S)$ contains the two points $\bar x \pm (\lambda-1) v^i$ or the two points $\bar x \pm \frac{\lambda r}{2(r+R)} v^i$, for every $i=1,\dots,k$.
Note that the lineality space of every minimal face of $Q_F^\lambda$ is $\lin F = \lin P$, thus for every split $S$ that satisfies $\lin P\nsubseteq \lin S$, we have $\conv (Q_F^\lambda \sm \int S) = Q_F^\lambda$.
Hence we now consider only splits with $\lin P \subseteq \lin S$.
To simplify notation, for fixed $S$ and $\lambda$ we define $T=\conv(Q_F^\lambda \sm \int S)$, omitting the dependence on $S$ and $\lambda$.

\paragraph{Case 1. Let $S$ be a split such that there exists a vector $\bar v \in \{v^1,\dots,v^k\}$ not in $\lin S$.}
In this case we show that $T$ contains the point $\bar x + (\lambda-1) v^i$ for every $i = 1,\dots,k$.
Symmetrically, $T$ will also contain the point $\bar x - (\lambda-1) v^i$ for every $i=1,\dots,k$.

Let $i\in \{1,\dots,k\}$ be such that $v^i \notin \lin S$.
As $v^i \in \Z^n$, it is easy to check that $\int S$ can contain at most one of the points $\bar x + \lambda v^i$ and $\bar x + (\lambda-1) v^i$.
Thus $T$ contains the point $\bar x + \lambda v^i$ or the point $\bar x + (\lambda-1) v^i$.
If $T$ contains $\bar x + \lambda v^i$, then it also contains $\bar x + (\lambda-1) v^i$, since the latter can be written as a convex combination of the points $\bar x + \lambda v^i$ and $\bar x$, which are both in $T$.
%As $\bar x \in F$, it follows that $T$ must contain $\bar x + (\lambda-1) v^i$.

Now let $i\in \{1,\dots,k\}$ be such that $v^i \in \lin S$.
If $\bar x +(\lambda-1)v^i \notin \int S$ we are done, thus we assume that $\bar x +(\lambda-1)v^i \in \int S$.
Since the three points $\bar x +\lambda v^i$, $\bar x \pm \lambda \bar v$ are in $Q_F^\lambda$, also their convex combinations $\bar x +(\lambda-1)v^i\pm\bar v$ are in $Q_F^\lambda$.
As $\bar v \in \Z^n$, $\bar v \notin \lin S$, and $\bar x +(\lambda-1)v^i \in \int S$, both points $\bar x +(\lambda-1)v^i\pm \bar v$ are not in $\int S$, and thus are in $T$.
Therefore also their convex combination $\bar x +(\lambda-1)v^i$ is in $T$.

\paragraph{Case 2. Let $S$ be a split such that $v^i \in \lin S$ for every $i =1,\dots,k$.}
In this case we show that $T$ contains the point $\bar x + \frac{\lambda r}{2(r+R)} v^i$ for every $i =1,\dots,k$.
Symmetrically, $T$ will also contain the point $\bar x - \frac{\lambda r}{2(r+R)} v^i$ for every $i=1,\dots,k$.

Let $\tilde v \in \{v^1,\dots,v^k\}$.
If $\bar x \notin \int S$, then also $\bar x+ \lambda \tilde v\notin S$, and the statement follows trivially, as $\bar x \in F$.
Thus we now assume that $\bar x \in \int S$, which implies that also $\bar x + \lambda \tilde v \in \int S$.

Since, by (i), $\relint(F+ L)$ is not contained in $\int S$, and since $F \cap \int S \neq \varnothing$, $F+ L$ is not contained in $S$.
As $v^i \in \lin S$ for every $i=1,\dots,k$, this implies that $F$ is not contained in $S$.
Therefore wlog $a x  \ge \beta$ is not valid for $F$, where $a \in \Z^n$, $\beta \in \Z$ are such that $S = \{x \in \R^n : \beta \le ax \le \beta+1\}$.
Since $F$ is integral, there exists a point in $F$ that satisfies $a x  \le \beta -1$.
By~\eqref{eq: dec}, such point can be written as
$$\sum_{i=1}^p\lambda^ig^i+\sum_{i=1}^q\mu^ih^i,$$
for nonnegative scalars $\lambda^1,\dots,\lambda^p$ and $\mu^1,\dots,\mu^q$ with $\sum_{i=1}^p\lambda^i=1$.

We now define a special point $w$ in $F$ that satisfies the inequality $a x  \le \beta -1$ and that is distant at most $R$ from $\bar x$.
If there exists a point in $\{g^1,\dots,g^p\}$ that satisfies $a x  \le \beta -1$, let $w$ be such point.
Otherwise, for every $i=1,\dots,p$, the integral point $g^i$ satisfies $a x  \ge \beta$, and so does the convex combination $\sum_{i=1}^p\lambda^ig^i$.
Therefore the scalar product of vectors $\sum_{i=1}^q\mu^ih^i$ and $a$ is strictly negative, implying that there exists a vector $h \in \{h^1,\dots,h^q\}$ such that the scalar product of $h$ and $a$ is strictly negative.
Define $w=\bar x+2h$.
As $h$ is in the recession cone of $F$, it follows that $w$ is in $F$.
Moreover, since $h$ is integral, $a w  \le \beta -1$.

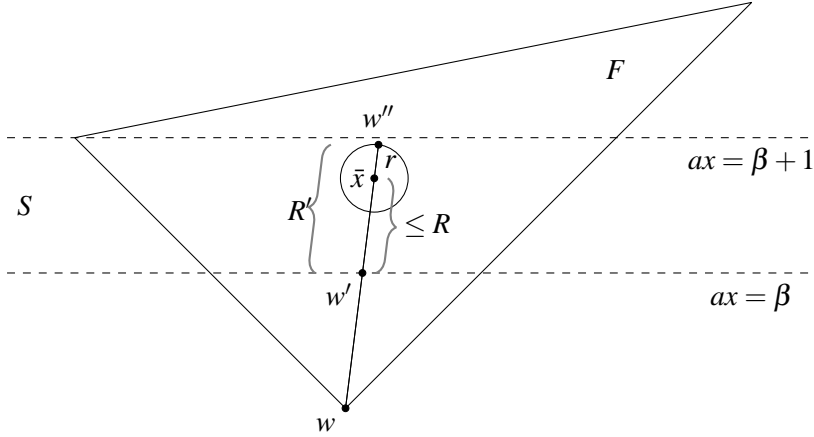
\begin{figure}[ht]
  \centering
\begin{tikzpicture}[scale=0.9,line join=bevel]
%\draw[help lines] (0,0) grid (10,6);
% split
\coordinate (s1) at (-1,2);
\coordinate (s2) at (11,2);
\coordinate (s3) at (-1,4);
\coordinate (s4) at (11,4);
\coordinate (h1) at (-1,3);
\coordinate (h2) at (11,3);
\draw[style= dashed, name path=split] (s1) -- (s2);
\draw[style= dashed] (s3) -- (s4);
\node[right] at (h1) {$S$};
\node[below] at ([xshift=-1cm]s2) {$ax= \beta$};
\node[below] at ([xshift=-1cm]s4) {$ax= \beta+1$};
% face F
\coordinate (a) at (4,0);
\coordinate (b) at (10,6);
\coordinate (c) at (0,4);
\draw (a) -- (b) -- (c) -- cycle;
\node at ([xshift=-2cm,yshift=-1cm]b) {$F$};
% line
\filldraw (a) circle [radius=0.05];
\node[below left] at (a) {$w$};
%\node[below right] at (a) {$G$};

\path[name path=half] (h1) -- (h2);
\path[name path=barxy] ([yshift=0.4cm]h1) -- ([yshift=0.4cm]h2);
\path[name path=diag] (a) -- (4.5,4);

\path[name intersections={of=diag and split,by={wone}}];
\filldraw (wone) circle [radius=0.05];
\node[below left] at (wone) {$w'$};

\path[name intersections={of=diag and barxy,by={barx}}];
\filldraw (barx) circle [radius=0.05];
\draw[name path=cir] (barx) circle [radius=0.5];
\node[left] at (barx) {$\bar x$};

\path[name intersections={of=diag and cir}];
\filldraw (intersection-1) circle [radius=0.05];
\node[above] at ([yshift=0.1cm]intersection-1) {$w''$};
\node[above right] at (barx) {$r$};
\draw[fill] (a) -- (intersection-1);
\draw[decoration={brace,mirror,amplitude=0.5em},decorate,thick,gray] ([xshift=0.2cm]wone) --  ([xshift=0.2cm]barx) node [black,midway,xshift=0.6cm] {$\le R$};
\draw[decoration={brace,amplitude=0.5em},decorate,thick,gray] ([xshift=-0.7cm]wone) --  ([xshift=-0.7cm]intersection-1) node [black,midway,xshift=-0.3cm] {$R'$};
\end{tikzpicture}
  \caption{Illustration of Case 2.}
    \label{fig: suff}
\end{figure}

Since $a w  \le \beta-1$, and $a \bar x  > \beta$, we can define $w'$ as the unique point in the intersection of the hyperplane $\{x\in\R^n : ax = \beta\}$ with the segment $[w,\bar x]$.
(See Fig.~\ref{fig: suff}.)
%We show that $T$ contains the point $w' + \frac{\lambda}{2} \tilde v$.
As $a \bar x < \beta+1$, it follows that the segment $[w,\bar x + \lambda \tilde v] \subseteq Q_F^\lambda$ contains a point $w' + \lambda' \tilde v$, with $\lambda' > \frac{\lambda}{2}$.
Thus the point $w' + \frac{\lambda}{2} \tilde v$ is in $Q_F^\lambda$, and in $T$.

We finally show that $T$ contains the point $\bar x + \frac{\lambda r}{2(r+R)} \tilde v$.
Let $w''$ be the intersection point of the line $\aff\{w,\bar x\}$ with the boundary of $B(\bar x,r)$ that does not lie in the segment $[w,\bar x]$.
The point $w''$ is in $F$, and the distance between $\bar x$ and $w''$ is $r$.
Let $R'$ be the distance between $w'$ and $w''$, and note that $r<R' \le R+r$.
Both points $w' + \lambda /2 \tilde v$ and $w''$ are in $T$, thus also is their convex combination $\frac{r}{R'}\left(w' + \frac{\lambda}{2} \tilde v\right) + \frac{R'-r}{R'}w'' = \bar x + \frac{\lambda r}{2R'} \tilde v$. As $R' \le R+r$, the point $\bar x + \frac{\lambda r}{2(r+R)} \tilde v$ is a convex combination of the latter point and $\bar x$, implying that $\bar x + \frac{\lambda r}{2(r+R)} \tilde v \in T$.

\section{Proof of necessity for bounded polyhedra}\label{sec:nec}

In this section we prove that if an integral polytope $P$ has infinite reverse split rank, then $F$ and $L$ satisfying conditions (i)--(ii) of Theorem~\ref{th:main} exist, while the case of an unbounded polyhedron will be considered in Sect.~\ref{sec:nec-unbounded}. We remark that if $P=\varnothing$ then its reverse split rank is finite, as this is the case even for the reverse CG rank (see \cite{CCT,ipco}). Therefore in this section we assume that $P\ne\varnothing$.
Also, we recall that for a polytope the condition $L\not\subseteq\lin P$ is equivalent to $L\ne\{0\}$.

In order to prove the necessity of conditions (i)--(ii), we need to extend the notion of relaxation and reverse split rank to {\em rational} polyhedra.
Indeed, when dealing with a non-full-dimensional integral polytope $P$ in Sect.~\ref{sec:non-full}, we will approximate $P$ with a non-integral full-dimensional polytope containing the same integer points as $P$.

Given a rational polyhedron $P\subseteq\R^n$, a {\em relaxation} of $P$ is a rational polyhedron $Q\subseteq\R^n$ such that $P\subseteq Q$ and $P\cap\Z^n=Q\cap\Z^n$. The reverse split rank of a rational polyhedron $P$ is defined as follows:
\[s^*(P)=\sup\{s(Q): \mbox{$Q$ is a relaxation of $P$}\}.\]

In the following we prove that if a nonempty rational polytope has infinite reverse split rank, then $F$ and $L$ satisfying conditions (i)--(ii) of Theorem~\ref{th:main} exist.

\subsection{Outline of the proof for full-dimensional polytopes}\label{sec:nec-outline}

Given a full-dimensional rational polytope $P\subseteq\R^n$ with $s^*(P)=+\infty$, we prove conditions (i)--(ii) of Theorem~\ref{th:main} under the assumption that the result holds for all ({\em possibly non-full-dimensional}) rational polytopes in $\R^{n-1}$.
(The case of a non-full-dimensional polytope in $\R^n$ will be treated in Sect.~\ref{sec:non-full}.)
Note that the theorem holds for $n=1$, as in this case $s^*(P)$ is always finite.
We also remark that if $P$ is bounded then every relaxation of $P$ is bounded, as every rational polyhedron has the same recession cone as its integer hull.

So let $P\subseteq\R^n$ be a full-dimensional rational polytope with $s^*(P)=+\infty$.
We now give a procedure that returns $F$ and $L$ satisfying the conditions of the theorem. We justify it and prove its correctness in the rest of this section.
We remark that at this stage the linear subspace returned by the procedure might be non-rational, but we will show in Sect.~\ref{sec:rationality} how to choose a rational subspace.
Also, we point out that the procedure below is not an ``executable algorithm'', but only a theoretical proof of the existence of $F$ and $L$.

\begin{enumerate}
\item
Fix a point $\bar x\in\int P$; choose a sequence $(Q_i)_{i\in\N}$ of relaxations of $P$ with $\sup_i s(Q_i)=+\infty$; initialize $k=1$, $L_0=\{0\}$, and $S=P$.
\item
Choose a sequence of points $(x_i)_{i\in\N}$ such that $x_i\in Q_i$ for all $i\in\N$ and $\sup_i\dist(x_i,S)=+\infty$;
let $v_i$ be the projection of $x_i-\bar x$ onto $L_{k-1}^\bot$;
define $\bar v$ as the limit of some subsequence of the sequence $\left(\frac{v_i}{\|v_i\|}\right)_{i\in\N}$ and assume wlog that this subsequence coincides with the original sequence; define $L_k=\langle L_{k-1},\bar v\rangle$.
\item
If $P+L_k$ is not contained in any split, then return $F=P$ and $L=L_k$, and stop; otherwise,
let $S$ be a split such that $P+L_k\subseteq S$, where $S=\{x\in\R^n:\beta\le ax\le \beta+1\}$.
\item
If there exists $M\in\R$ such that $Q_i\subseteq\{x\in\R^n:\beta-M\le ax\le \beta+M\}$ for every $i\in\N$, then choose $j\in\{0,1\}$ such that $P^j:=P\cap\{x\in\R^n:ax=\beta+j\}$ has infinite reverse split rank when viewed as a polytope in the affine space $H=\{x\in\R^n:ax=\beta+j\}$; since $H$ is a rational subspace and we assumed that the result holds in dimension $n-1$, there exist $F$ and $L$ satisfying conditions (i)--(ii) of the theorem with respect to the space $H$; return $F$ and $L$, and stop.
Otherwise, if no $M$ as above exists, set $k\leftarrow k+1$ and go to 2.
\end{enumerate}

In order to prove the correctness of the above procedure, we will show the following:
\begin{enumerate}[(a)]
\item
in step~2, a sequence $(x_i)_{i\in\N}$ and a vector $\bar v$ as required can be found;
\item
the procedure terminates (either in step~3 or step~4);
\item
if the procedure terminates in step~4, then there exists $j\in\{0,1\}$ such that $P^j$ has infinite reverse split rank in the affine space $H=\{x\in\R^n:ax=\beta+j\}$, and the output is correct;
\item
if the procedure terminates in step~3, then the output is correct.
\end{enumerate}

\subsection{Proof of (a)}

We prove that at every execution of step~2 a sequence $(x_i)_{i\in\N}$ and a vector $\bar v$ as required can be found.

Consider first the iteration $k=1$; in this case, $S=P$. Since $\sup_i s(Q_i)=+\infty$, we also have $\sup_i r(Q_i)=+\infty$. By Corollary~\ref{cor:ub-chv} applied to the integral polytope $\conv(P\cap\Z^n)$, there is no bounded set containing every $Q_i$ for $i\in\N$. Then there is a sequence of points $(x_i)_{i\in\N}$ such that $x_i\in Q_i$ for every $i\in\N$ and $\sup_i\dist(x_i,P)=+\infty$.
As $L_0=\{0\}$, for $k=1$ the definition of $v_i$ given in step~2 reduces to $v_i=x_i-\bar x$ for $i\in\N$. Since every vector $\frac{v_i}{\|v_i\|}$ belongs to the unit sphere, which is a compact set, the sequence $\left(\frac{v_i}{\|v_i\|}\right)_{i\in\N}$ has a subsequence converging to some unit-norm vector $\bar v$.

Assume now that we are at the $k$-th iteration ($k\ge2$). Then the algorithm has determined a split $S\subseteq\R^n$ such that $P+L_{k-1}\subseteq S=\{x\in\R^n:\beta\le ax\le \beta+1\}$. Furthermore, we know that there is no $M\in\R$ such that $Q_i\subseteq\{x\in\R^n:\beta-M\le ax\le \beta+M\}$ for every $i\in\N$ (see step~4).
This implies that there is a sequence of points $(x_i)_{i\in\N}$ such that $x_i\in Q_i$ for $i\in\N$ and $\sup_i\dist(x_i,S)=+\infty$.
For $i\in\N$, let $v_i$ be the projection of the vector $x_i-\bar x$ onto the space $L_{k-1}^\bot$. Note that, for $i$ large enough, $x_i-\bar x\notin L_{k-1}$, as $\bar x+L_{k-1}\subseteq S$ and $\sup_i\dist(x_i,S)=+\infty$; thus $v_i\ne0$ for $i$ large enough.
Since the elements of the sequence $\left(\frac{v_i}{\|v_i\|}\right)_{i\in\N}$ belong to the intersection of $L_{k-1}^\bot$ with the unit sphere, and this intersection gives a compact set, there is a subsequence converging to some unit-norm vector belonging to $L_{k-1}^\bot$, which we call $\bar v$.

\subsection{Proof of (b)}

In order to show that the procedure terminates after a finite number of iterations, it is sufficient to observe that at every iteration in step~2 we select a vector $\bar v\in L_{k-1}^\bot$, thus the dimension of $L_k=\langle L_{k-1},\bar v\rangle$ is $k$.
In particular, the procedure terminates after at most $n$ iterations, as for $k=n$ no split $S$ can be found in step~3.

\subsection{Proof of (c)}\label{sec:c}

We now prove that if the procedure terminates in step~4, then there exists $j\in\{0,1\}$ such that $P^j$ has infinite reverse split rank when viewed as a polytope in the affine space $\{x\in\R^n:ax=\beta+j\}$, and the output is correct.

Since $Q_i\subseteq\{x\in\R^n:\beta-M\le ax\le \beta+M\}$ for every $i\in\N$, by Lemma~\ref{lem:ub-chv} there exists a number $N$ such that, for each $i\in\N$, $N$ iterations of the CG closure operator (hence, also of the split closure operator) applied to $Q_i$ are sufficient to obtain a relaxation of $P$ contained in $S$. For $i\in\N$, let $\widetilde Q_i$ be the relaxation of $P$ obtained this way. Then we have $\sup_i s(\widetilde Q_i)=+\infty$.

Recall that $P^0$ and $P^1$ are the faces of $P$ induced by equations $ax=\beta$ and $ax=\beta+1$, respectively. Similarly, for $i\in\N$, let $\widetilde Q^0_i$ and $\widetilde Q^1_i$ be the faces of $\widetilde Q_i$ induced by equations $ax=\beta$ and $ax=\beta+1$, respectively.
Since $\widetilde Q_i\subseteq S$, by Lemma~\ref{lem:split-faces} we have $s(\widetilde Q_i)\le \max\{s(\widetilde Q_i^0),s(\widetilde Q^1_i)\}+1$. Then there exists $j\in\{0,1\}$ such that $\sup_is(\widetilde Q_i^j)=+\infty$. Since every relaxation $\widetilde Q_i^j$ is contained in the affine space $H=\{x\in\R^n:ax=\beta+j\}$, we have $s^*(P^j)=+\infty$ with respect to the ambient space $H$ (which is equivalent to $\R^{n-1}$ under some unimodular transformation). Let $H^*$ be the translation of $H$ passing through the origin. Since $H$ is a rational space of dimension $n-1$, by induction there exist a face $F$ of $P^j$ and a nonzero linear subspace $L\subseteq H^*$ satisfying conditions (i)--(ii) of Theorem~\ref{th:main} for $P^j$: specifically, $\relint(F+L)$ is not contained in the interior of any $(n-1)$-dimensional split in the affine space $H$, and $G+L$ is relatively lattice-free for every face $G$ of $P^j$ containing $F$.

We show that $F$ and $L$ satisfy conditions (i)--(ii) for $P$, too. First, note that $F$ is a face of $P$ and $L$ is a nonzero linear subspace of $\R^n$.
To prove (i), assume by contradiction that there is an $n$-dimensional split $T$ such that $\relint(F+L)\subseteq\int T$. Then, as $F+L$ is contained in the boundary of $S$, we have $\lin T\ne\lin S$. This implies that $T\cap H$ is contained in some $(n-1)$-dimensional split $U$ living in $H$. But then, with respect to the ambient space $H$, we would have $\relint(F+L)\subseteq\int U$, a contradiction.

To prove (ii), let $G$ be a face of $P$ containing $F$. If $G\subseteq P^j$, then $G$ is a face of $P^j$ and thus $G+L$ is relatively lattice-free by induction. So we assume that $G\not\subseteq P^j$. Since $G\subseteq P\subseteq S$, this implies that $G$ contains some points in $\int S$, and thus $\relint G\subseteq\int S$. Since $L\subseteq\lin S$, this yields $\relint(G+L)\subseteq\int S$, hence $G+L$ is relatively lattice-free.

\subsection{Proof of (d)}\label{sec:d}

We now prove that if the procedure terminates in step~3, then the output is correct. Note that it is sufficient to prove that $P+L_k$ is lattice-free at every iteration of the algorithm.

The subspace $L_1$ is constructed following the same procedure as in the proof of Theorem~\ref{th:CG} given in \cite[Sect.~3.2]{ipco}. Therefore, with the same arguments as in \cite{ipco}, one proves that $P+L_1$ is lattice-free.

We now assume that $k\in\{2,\dots,n\}$. Recall that $L_k=\langle L_{k-1},\bar v\rangle$ and $\bar v\in L^\bot_{k-1}$.
Suppose by contradiction that there is an integer point $\bar z\in\int(P+L_k)=\int P+L_k$.
First of all we show that $\bar z$ can be taken sufficiently far from $P$ (we will specify later how far it should be taken).
To see this, choose any integer point $\bar z\in\int P+L_k$ and apply Lemma~\ref{lem:integer-point} to the affine subspace $\bar z+\langle\bar v\rangle$ (which is a translation of a linear subspace by an integer point), with $\delta$ small enough. This guarantees the existence of integer points in $\int P+L_k$ that lie arbitrarily far from $P$.

Since $\bar z\in\int P+L_k$, there exists a vector $u\in L_{k-1}$ such that $z_0:=\bar z+u\in\int P+\langle \bar v\rangle$. Let $x_0 \in \int P$ be such that $x_0=z_0-d \bar v$, where wlog we can assume $d>0$.
Note that since $\|\bar v\|=1$, $d=\|z_0-x_0\|$.
Let $r>0$ be such that $B(x_0,r)\subseteq P$. Furthermore, denote by $\pi:\R^n\to L_{k-1}^\bot$ the orthogonal projection onto $L_{k-1}^\bot$.

Recall that there is a split $S$ such that $P+L_{k-1}\subseteq S$ (step~3 of the previous iteration). Define $H=\lin S$ and $H_0=z_0+H$. By starting the above construction with a point $\bar z$ sufficiently far from $P$, we can assume wlog that $H_0$ does not intersect $P$.

We need the following lemma (recall that $x_0\in\int P$).

\begin{myclaim}\label{cl:limit}
For every $M'>0$ and $\varepsilon>0$, there exist an index $i\in\N$ and points $y_1,\dots,y_k$ satisfying $y_t\in Q_i\cap(x_0+L_t)$, $\dist(y_t,x_0+L_{t-1})\ge M'$ for $t=1,\dots,k$, and $\phi(\pi(y_k-x_0),\bar v)\le\varepsilon$.
\end{myclaim}

\begin{proof}
We proceed by induction on $k$: we assume that the property of the lemma holds when $k$ is replaced with $k-1$ (i.e., at the previous iteration of the algorithm), and show that it also holds at the current iteration. We remark that the base case $k=1$ does not need to be treated separately. (See Fig.~\ref{fig:main-base} to follow the proof.)

Fix any $M'>0$ and $\varepsilon>0$.
By induction, there exist an index $i$ and points $y_1,\dots,y_{k-1}$ such that $y_t\in Q_i\cap(x_0+L_t)$ and $\dist(y_t,x_0+L_{t-1})\ge M'$ for $t=1,\dots,k-1$. Note that the existence of such an index $i$ implies the existence of infinitely-many such indices. (To see this, one has to reapply the lemma with a sufficiently large $M'$: since $Q_i$ is bounded, if $M'$ is large enough then a different index $i'$ must exist.)
Now we need to find an additional point $y_k\in Q_i\cap(x+L_k)$ such that $\dist(y_k,x_0+L_{k-1})\ge M'$.

Let $r>0$ be such that $B(x_0,r)\subseteq P$ and define $d=\|x_0-\bar x\|$. Denote again by $\pi:\R^n\to L_{k-1}^\bot$ the orthogonal projection onto $L_{k-1}^\bot$ (recall that $\bar v\in L_{k-1}^\bot$).
The choices of the sequence $(x_i)_{i\in\N}$ and the vector $\bar v$ made in step~2 of the algorithm imply that, for $i\in\N$ large enough, the norm of $\pi(x_i-\bar x)$ can be made arbitrarily large and the angle $\phi(\pi(x_i-\bar x),\bar v)=\phi(v_i,\bar v)$ can be made arbitrarily small. Thus we can assume that
\begin{equation}\label{eq:D}
D:=\|\pi(x_i-\bar x)\|\ge\frac{r}{2(M'(r+d+1)+d)},
\end{equation}
\[\alpha:=\phi(\pi(x_i-\bar x),\bar v)=\phi(v_i,\bar v)\le\min\left\{\frac{\pi}{3},\,\arcsin\left(\frac{r}{2(M'(r+d+1)+d)}\right)\right\}.\]
By replacing $x_i$ with a suitable point in the line segment $[\bar x,x_i]$, we can assume that \eqref{eq:D} holds at equality: $D=\frac{2(M'(r+d+1)+d)}{r}$.

Let $w_i$ be the orthogonal projection of $x_i$ onto the affine space $x_0+L_{k-1}^\bot$.
We claim that there exists $x'\in B(x_0,r)\cap(x_0+L_k^\bot)$ such that $\|x'-x_0\|=r$ and $[x',w_i]\cap(x_0+\langle \bar v\rangle)$ contains a (single) point, which we call $z$. To see this, consider the set obtained as the convex hull of $x_i$ and $B(x_0,r)\cap(x_0+L_k^\bot)$. This set intersects the line $x_0+\langle\bar v\rangle$ in a segment. The point of this segment at maximum distance from $x_0$ is the point $x'$ satisfying the requirements.

Observe that the orthogonal projections of $x_0$ and $x'$ onto $w_i+\langle\bar v\rangle$ coincide; let us call $u$ this common projected point.
Then
\[\|u-x'\|\le r+d+D\sin\alpha\leq r+d+\frac{Dr}{2(M'(r+d+1)+d)}=r+d+1\]
and
\[\|u-w_i\|\ge D\cos\alpha-d\ge D/2-d.\]
Since the two triangles with vertices respectively $w_i,u,x'$ and $z,x_0,x'$ are similar, we deduce that
\begin{equation}\label{eq:M'}
\|z-x_0\|\ge \frac{D/2-d}{r+d+1}\cdot r\ge\frac{M'(r+d+1)}{r+d+1}=M'.
\end{equation}
The construction of $z$ shows that there exists $y_k\in Q_i\cap(x_0+L_k)$ whose projection onto $x_0+L_{k-1}^\bot$ is $z$, hence
\[\dist(y_k,x_0+L_{k-1})\ge \dist(z,x_0)\geq M'.\]

To show the last part of Claim~\ref{cl:limit}, note that if we choose $i$ sufficiently large then the angle $\phi(\pi(y_k-\bar x),\bar v)$ can be made arbitrarily small. Since the norm of $\pi(y_x-\bar x)$ can be made arbitrarily large, this implies that also the angle $\phi(\pi(y_k-x_0),\bar v)$ can be made arbitrarily small. (In other words, the angles $\phi(\pi(y_k-\bar x),\bar v)$ and $\phi(\pi(y_k-x_0),\bar v)$ are almost the same for large $i$.)
\end{proof}

\begin{figure}
\begin{center}
\psset{unit=9mm}
\psset{linewidth=.5pt}
\begin{pspicture}(13.8,5.3)
\put(0.5,1.5){\circle*{.1}}
\put(0.3,1.1){$\underline{\bar x}$}
\psline(0,1.5)(11.5,1.5)
\put(11.8,1.4){$\underline{\bar x}+\langle\bar v\rangle$}
\put(2,4){\circle*{.1}}
\put(2.1,3.7){$x_0$}
\psline(0,4)(11.5,4)
\put(11.8,3.9){$x_0+\langle\bar v\rangle$}
\put(2,4.8){\circle*{.1}}
\put(2.1,4.9){$x'$}
\psline{<->}(2,4.08)(2,4.72)
\put(1.7,4.3){$r$}
\put(10.5,0.5){\circle*{.1}}
\put(10.1,0.1){$w_i,\underline{x_i}$}
\psline[linestyle=dashed](0.5,1.5)(10.5,0.5)
\put(3.6,1.25){$\alpha$}
\psline(2,4.8)(10.5,0.5)
\put(3.55,4){\circle*{.1}}
\put(3.6,4.2){$z,\underline{y_k}$}
\psline[linestyle=dashed](2,4)(2,.5)
\psline(0,0.5)(11.5,0.5)
\put(11.8,0.4){$w_i+\langle\bar v\rangle$}
\put(2,.5){\circle*{.1}}
\put(1.9,.1){$u$}
\end{pspicture}
\caption{Illustration of the proof of Claim~\ref{cl:limit}. The affine space $x_0+L_{k-1}^\bot$ is represented. Underlined symbols indicate points that do not necessarily belong to $x_0+L_{k-1}^\bot$; in other words, their orthogonal projection onto $x_0+L_{k-1}^\bot$ is represented. Note that $x_i$ may also lie ``above'' the line $\bar x+\langle \bar v\rangle$ (and even ``above'' the line $x_0+\langle\bar v\rangle$).}
\label{fig:main-base}
\end{center}
\end{figure}
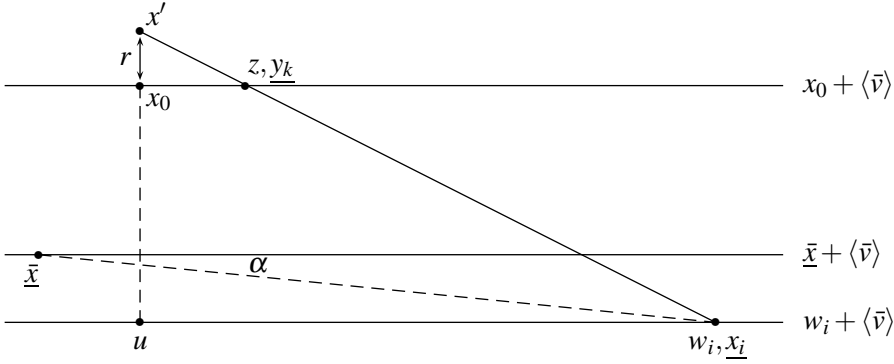

We first apply Lemma~\ref{lem:very-far-things} with $\delta=r/8$ and $k-1$ in place of $k$, and obtain $M>0$ such that the condition of the lemma is satisfied, i.e., for every $x\in L_{k-1}$ and for every $y_1,\dots,y_{k-1}$ satisfying $y_t\in x+L_t$ and $\dist(y_t,x+L_{t-1})\ge M$ for $t=1,\dots,k-1$, one has
$\dist(\conv(x,y_1,\dots,y_{k-1}), H\cap\Z^n)\leq \delta$. Define $M'=\max\{2M,2d\}$.
Now, by Claim~\ref{cl:limit}, there exist $i\in\N$ and points $y_1,\dots,y_k$ satisfying $y_t\in Q_i\cap(x_0+L_t)$ and $\dist(y_t,x_0+L_{t-1})\ge M'$ for $t=1,\dots,k$ (see Fig.~\ref{fig:lf}).
Now let $w$ denote the unit-norm vector which is orthogonal to $H$ and forms an acute angle with $\bar v$ (recall that $\bar v\notin H$, thus $\bar v$ and $w$ cannot be orthogonal) and define $\alpha=\phi(\bar v,w)$.
Again because of Claim~\ref{cl:limit}, we can enforce the condition
\begin{equation}\label{eq:beta}
\beta:=\phi(\pi(y_k-x_0),\bar v)\le\arctan\left(\tan\alpha+\frac r {8d}\right)-\alpha
\end{equation}
(see Fig.~\ref{fig:alpha}). Note that the value on the right-hand-side of \eqref{eq:beta} is nonnegative, as $0\le\alpha<\pi/2$.

For $\rho > 0$, define $B'(\rho)=B(0,\rho)\cap L_{k-1}^\bot\cap H$. For $t=1,\dots,k-1$, let $\tilde y_t$ be the midpoint of the segment $[x_0,y_t]$.
Note that
\[\begin{split}
Q_i & \supseteq\conv\left(B(x_0,r)\cup\{y_1,\dots,y_{k-1}\}\right)\\
&\supseteq C:=\conv\left(x_0+B'(r/2),\tilde y_1+B'(r/2),\dots,\tilde y_{k-1}+B'(r/2)\right).
\end{split}\]

Let $x'$ be the unique point in $[x_0,y_k]\cap H_0$, and, for $i=1,\dots,k-1$, let $y'_t$ be the unique point in $[\tilde y_t,y_k]\cap H_0$. Since $\dist(y_k,x_0+L_{k-1})\ge M'\ge 2d\ge 2\dist(x_0+L_{k-1},x'+L_{k-1})$,
\begin{equation}\label{eq:C'}
\conv(C,y_k)\cap H_0\supseteq C':=\conv(x'+B'(r/4),y'_1+B'(r/4),\dots,y'_k+B'(r/4)).
\end{equation}
Moreover, as $B(x_0,r)\subseteq P\subseteq Q_i$ and $B(\tilde y_t,r/2)\subseteq Q_i$ for $t=1,\dots,k-1$, we have
\begin{equation}\label{eq:balls}
B(x',r/2)\subseteq Q_i \qquad \mbox{and}\qquad B(y'_t,r/4)\subseteq Q_i\quad \mbox{for $t=1,\dots,k-1$}.
\end{equation}

Let $x''$ be the projection of $x'$ onto the space $z_0+L_{k-1}$.
We claim that
\[\|x''-x'\|=\|\pi(x''-x')\|\le d\tan(\alpha+\beta)-d\tan\alpha\le r/8\]
(see again Fig.~\ref{fig:alpha}).
The equality holds because $x''-x'\in L_{k-1}^\bot$ by construction; the first inequality describes the worst case (which is the one depicted in the figure), i.e., when $\|\pi(x''-x')\|$ is as large as possible;
%in fact, when $\beta$ and $d$ are fixed, the value of $\|\pi(x''-x')\|$ becomes larger and larger when $\pi(x')$ is further and further from $x_0$;
the last bound follows from \eqref{eq:beta}.

Now define $y''_1,\dots,y''_{k-1}$ as the orthogonal projections of $y'_1,\dots,y'_{k-1}$ onto $z_0+L_{k-1}=\bar z+L_{k-1}$. Note that $y''_1,\dots,y''_{k-1}$ are obtained by translating $y'_1,\dots,y'_{k-1}$ by vector $x''-x'$. By the definition of $C'$ given in \eqref{eq:C'}, $y''_1,\dots,y''_{k-1}$ still belong to $C'$.
One verifies that $y''_t\in x''+L_t$ and $\dist(y''_t,x''+L_{t-1})\ge M'/2\ge M$ for $t=1,\dots,k-1$.
Since $\bar z+L_{k-1}$ is a translation of $L_{k-1}$ by an integer vector, by the choice of $M$ given by Lemma~\ref{lem:very-far-things} there is an integer point $p\in \bar z+H=H_0$ at distance at most $\delta=r/8$ from the set $\conv(x'',y''_1,\dots,y''_{k-1})$.

We claim that $p\in Q_i$. To see this, first observe that
\begin{equation}\label{eq:p}
\begin{split}
\dist(p,\conv(x',y'_1,\dots,y'_{k-1}))&\le\dist(p,\conv(x'',y''_1,\dots,y''_{k-1}))+\|x''-x'\|\\
&\le \frac r8 + \frac r8=\frac r4.
\end{split}
\end{equation}
Now from \eqref{eq:balls} we obtain that $\conv(x',y'_1,\dots,y'_{k-1})+B(0,r/4)\subseteq Q_i$ and thus, by \eqref{eq:p}, $p\in Q_i$. This is a contradiction, as $p$ is an integer point in $Q_i\setminus P$ ($p$ does not belong to $P$ because $p\in H_0$ and $H_0\cap P=\varnothing$ by assumption).

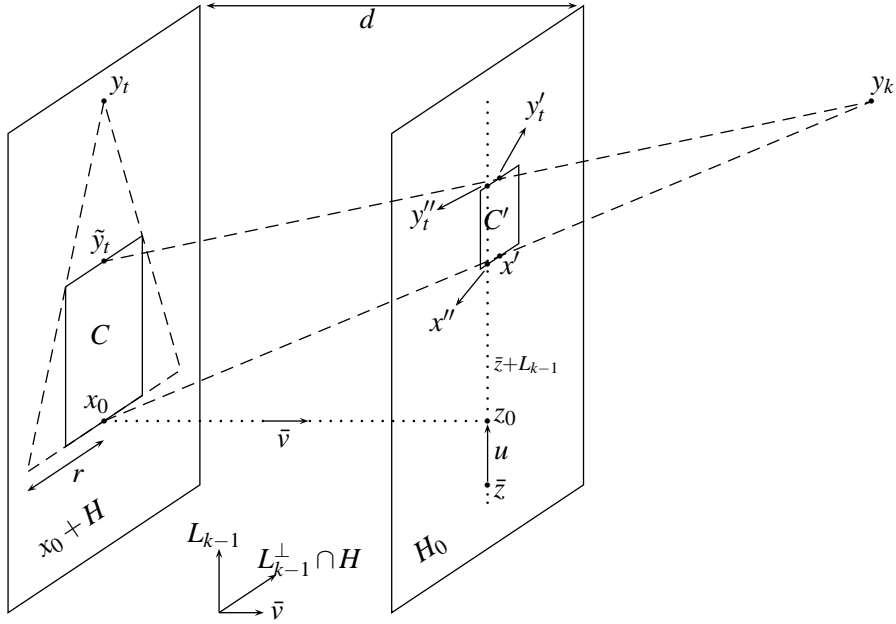
\begin{figure}
\begin{center}
\psset{unit=8.5mm}
\psset{linewidth=.5pt}
\begin{pspicture}(14,9.5)
\pspolygon(0,0)(0,7.5)(3,9.5)(3,2)
\rput{33.69}(1,1.3){$x_0+H$}
\put(1.5,3){\circle*{.1}}
\put(1.2,3.2){$x_0$}
\put(1.5,8){\circle*{.1}}
\put(1.6,8.2){$y_t$}
\pspolygon[linestyle=dashed](.3,2.2)(1.5,8)(2.7,3.8)
\psline{<->}(.3,1.9)(1.5,2.7)
\put(1,2.05){$r$}
\put(1.5,5.5){\circle*{.1}}
\put(1.3,5.7){$\tilde y_t$}
\put(1.3,4.2){$C$}
\pspolygon(0.9,2.6)(2.1,3.4)(2.1,5.9)(.9,5.1)
\pspolygon(6,0)(6,7.5)(9,9.5)(9,2)
\rput{33.69}(6.6,1){$H_0$}
\psline[linestyle=dotted,linewidth=1pt](1.5,3)(4,3)
\psline[linestyle=dotted,linewidth=1pt](4.7,3)(7.5,3)
\psline{->}(4,3)(4.7,3)
\put(4.2,2.6){$\bar v$}
\psline{->}(3.3,0)(3.3,1)
\put(2.8,1.1){$L_{k-1}$}
\psline{->}(3.3,0)(4.2,.6)
\put(3.9,0.7){$L_{k-1}^\bot\cap H$}
\put(7.5,3){\circle*{.1}}
\put(7.6,3){$z_0$}
\put(7.5,2){\circle*{.1}}
\put(7.6,1.8){$\bar z$}
\psline[linestyle=dotted,linewidth=1pt](7.5,1.7)(7.5,2)
\psline{->}(7.5,2.05)(7.5,2.95)
\put(7.6,2.4){$u$}
\psline[linestyle=dotted,linewidth=1pt](7.5,3)(7.5,8)
\put(13.5,8){\circle*{.1}}
\put(13.5,8.2){$y_k$}
\psline[linestyle=dashed](1.5,3)(13.5,8)(1.5,5.5)
\put(7.6875,5.578125){\circle*{.1}}
\put(7.6875,6.803125){\circle*{.1}}
\pspolygon(7.3875,5.378125)(7.9875,5.778125)(7.9875,7.003125)(7.3875,6.603125)
\put(7.45,6){$C'$}
\put(7.5,5.453125){\circle*{.1}}
\put(7.5,6.678125){\circle*{.1}}
\put(7.7,5.3){$x'$}
\put(6.6,4.4){$x''$}
\put(8.1,7.8){$y'_t$}
\put(6.3,6.1){$y''_t$}
\psline{->}(7.45,5.35)(7,4.8)
\psline{->}(7.7,6.9)(8.1,7.6)
\psline{->}(7.4,6.67)(6.7,6.3)
%\psline(7.5,5.453125)(7.5,6.678125)
\put(7.6,3.8){$\scriptstyle\bar z+L_{k-1}$}
\psline{<->}(3.1,9.5)(8.9,9.5)
\put(5.5,9.15){$d$}
\psline{->}(3.3,0)(4,0)
\put(4.1,-0.1){$\bar v$}
\end{pspicture}
\caption{Illustration of the proof of (d).}
\label{fig:lf}
\end{center}
\end{figure}

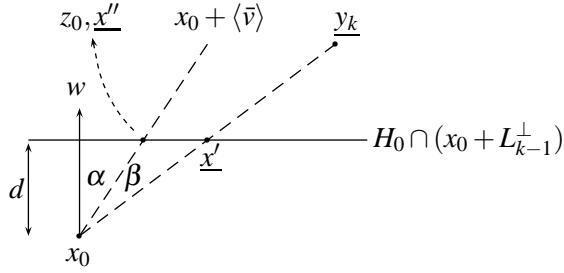
\begin{figure}
\begin{center}
\psset{unit=8.5mm}
\psset{linewidth=.5pt}
\begin{pspicture}(8.8,4.2)
\put(1,0.5){\circle*{.1}}
\psline(0.2,2)(5.5,2)
\put(5,3.5){\circle*{.1}}
\put(2,2){\circle*{.1}}
\psline[linestyle=dashed](1,0.5)(3,3.5)
\psline{->}(1,0.5)(1,2.5)
\put(3,2){\circle*{.1}}
\psline[linestyle=dashed](1,0.5)(5,3.5)
\put(2.5,3.8){$x_0+\langle\bar v\rangle$}
\put(0.8,2.7){$w$}
\put(5,3.8){$\underline{y_k}$}
\put(0.8,0.1){$x_0$}
\put(2.9,1.6){$\underline{x'}$}
\put(5.6,1.9){$H_0\cap(x_0+L_{k-1}^\bot)$}
\put(1.1,1.3){$\alpha$}
\put(1.7,1.3){$\beta$}
\psline[linearc=2,linestyle=dashed,dash=2pt 2pt]{->}(1.85,2.15)(1.4,2.6)(1.2,3.6)
\put(0.7,3.8){$z_0,\underline{x''}$}
\psline{<->}(0.2,0.5)(0.2,1.95)
\put(-0.1,1.1){$d$}
\end{pspicture}
\caption{Illustration of the proof of (d). The space $x_0+L_{k-1}^\bot$ is represented. Underlined symbols indicate points that do not necessarily belong to $x_0+L_{k-1}^\bot$; in other words, their orthogonal projection onto $x_0+L_{k-1}^\bot$ is represented.}
\label{fig:alpha}
\end{center}
\end{figure}

\subsection{Rationality of $L$}\label{sec:rationality}

As mentioned in Sect.~\ref{sec:nec-outline}, our procedure might return a non-rational linear subspace $L$. Note that this cannot be the case if the procedure terminates in step~4, as in this case the rationality of $L$ follows from the fact that we assumed that the theorem holds in $\R^{n-1}$. Therefore we now assume that the procedure terminates in step~3, and show that we can replace $L$ with a suitable nonzero rational linear subspace $\widetilde L$ and still have conditions (i)--(ii) fulfilled.

Since $P+L$ is full-dimensional, as discussed in Sect.~\ref{sec:lf} we have $P+L \subseteq \widetilde P+\widetilde L$, where $\widetilde P$ is a polytope and $\widetilde L$ is a rational linear subspace. Moreover, $\widetilde L\ne\{0\}$, as it contains $L$. Since we are assuming that the procedure terminates in step~3 (thus $F=P$), conditions (i)--(ii) are satisfied if $L$ is replaced with $\widetilde L$.

\subsection{The non-full-dimensional case}\label{sec:non-full}

The proof of the necessity of Theorem~\ref{th:main} given above covers the case of a full-dimensional rational polytope $P\subseteq\R^n$, assuming the result true both for full-dimensional and non-full-dimensional rational polytopes in $\R^{n-1}$. We now deal with the case of a non-full-dimensional polytope in $\R^n$.
For this purpose, we will take a non-full-dimensional polytope $P$ and make it full-dimensional by ``growing'' it along directions orthogonal to its affine hull. This will be done in such a way that no integer point is added to $P$. The idea is then to use the proof of the full-dimensional case given above.
We remark that even if we start from an integral polytope $P$, the new polytope that we construct will not be integral. This is why at the beginning of Sect.~\ref{sec:nec} we extended the notion of reverse split rank to rational polyhedra.

Note that if $P_I$ is the convex hull of integer points in a rational polytope $P$, it is not true (in general) that $s^*(P_I)=+\infty$ implies $s^*(P)=+\infty$. However, the key fact underlying our approach is the following:
\begin{quote}\em
Given a non-full-dimensional rational polytope $P$ with $s^*(P)=+\infty$, it is possible to ``enlarge'' $P$ and obtain a full-dimensional polytope $P'$ containing the same integer points as $P$, in such a way that $s^*(P')=+\infty$.
\end{quote}

Now, let $P$ be a $d$-dimensional rational polytope $P$, where $d<n$.
Assume that $s^*(P)=+\infty$. By applying a suitable unimodular transformation, we can assume that $\aff P=\R^d\times\{0\}^{n-d}$.

Given a rational basis $\{b_{d+1},\dots,b_n\}$ of the subspace $(\aff P)^\bot=\{0\}^d\times\R^{n-d}$, a rational point $\bar x\in\relint P$, and a rational number $\varepsilon>0$, we define
\[P(\bar x,\varepsilon)=\conv(P,\bar x+\varepsilon b_{d+1},\dots,\bar x+\varepsilon b_n);\]
we do not write explicitly the dependence on vectors $b_{d+1},\dots,b_n$, as they will be soon fixed. Note that $P(\bar x,\varepsilon)$ is a full-dimensional rational polytope.

We can now present the procedure that finds $F$ and $L$ as required. Recall that we are assuming by induction that the theorem is true for both full-dimensional and non-full-dimensional rational polytopes in $\R^{n-1}$.

\begin{enumerate}
\item[0.]
Let $\omega:=\omega(n-d,1)$ be the constant of Lemma~\ref{lem:flat}.
Choose a sequence $(Q_i)_{i\in\N}$ of relaxations of $P$ such that $\sup_i s(Q_i)=+\infty$.
For $i\in\N$, let $\widehat Q_i$ be the orthogonal projection of $Q_i$ onto the space $(\aff P)^\bot$, which we identify with $\R^{n-d}$.
If there exists an infinite subsequence of indices $i_1,i_2,\dots$ such that the lattice width of every polyhedron $\widehat Q_{i_t}$ in $\R^{n-d}$ is at most $\omega$, then $s^*(P)=+\infty$ also when we view $P$ as a polyhedron in $\R^{n-1}$. In this case, return $F$ and $L$ by induction, and stop.
\item
Fix a rational point $\bar x\in\relint P$; choose a rational basis $\{b_{d+1},\dots,b_n\}$ of $(\aff P)^\bot$, a rational number $\varepsilon>0$, and redefine the sequence of rational polyhedra $(Q_i)_{i\in\N}$ so that:
\begin{enumerate}[\upshape(a)]
\item
$P(\bar x,\varepsilon)$ has the same integer points as $P$,
\item
$Q_i$ is a relaxation of $P(\bar x,\varepsilon)$ (and thus of $P$) for every $i\in\N$,
\item
$\sup_i s(Q_i)=+\infty$;
\end{enumerate}
initialize $k=1$, $L_0=\{0\}$, and $S=P(\bar x,\varepsilon)$.
\item
Choose a sequence of points $(x_i)_{i\in\N}$ such that $x_i\in Q_i$ for all $i\in\N$ and $\sup_i\dist(x_i,S)=+\infty$;
let $v_i$ be the projection of $x_i-\bar x$ onto $L_{k-1}^\bot$;
define $\bar v$ as the limit of some subsequence of the sequence $\left(\frac{v_i}{\|v_i\|}\right)_{i\in\N}$ and assume that this subsequence coincides with the original sequence; define $L_k=\langle L_{k-1},\bar v\rangle$.
\item
If, for every strictly positive rational number $\varepsilon'\le\varepsilon$, $P(\bar x,\varepsilon')+L_k$ is not contained in any split, then choose a rational subspace $L\supseteq L_k$ such that $P(\bar x,\varepsilon)+L$ is lattice-free, return $F=P$ and $L$, and stop; otherwise,
let $S=\{x\in\R^n:\beta\le ax\le \beta+1\}$ be a split such that $P(\bar x,\varepsilon')+L_k\subseteq S$ for some strictly positive rational number $\varepsilon'\le\varepsilon$, and update $\varepsilon\leftarrow\varepsilon'$.
\item
If there exists $M\in\R$ such that $Q_i\subseteq\{x\in\R^n:\beta-M\le ax\le \beta+M\}$ for every $i\in\N$, then choose $j\in\{0,1\}$ such that $P^j:=P\cap\{x\in\R^n:ax=\beta+j\}$ has infinite reverse split rank (when viewed as a polytope in the affine space $\{x\in\R^n:ax=\beta+j\}$), then $F$ and $L$ exist by induction; return $F$ and $L$, and stop.
Otherwise, set $k\leftarrow k+1$, and go to 2.
\end{enumerate}

%We remark that $\bar x$ does not lie in the interior of $P$. In fact, though having $\bar x$ in the interior of $P$ is helpful in the proof, in the construction this point is only used to determine limit directions. It is clear that all limit directions remain the same if we choose $\bar x$ not in $\int(P)$. In the extension we give below, $\bar x$ will not lie in the interior of $P$ (this helps make the presentation simpler).

The fact that step 2 can be executed follows from the same argument given for the full dimensional case (see Sect. 5.2).
The following additional facts, which we prove below, imply the correctness of the procedure:
\begin{itemize}
\item
in step 0, if there exists an infinite subsequence of indices $i_1,i_2,\dots$ such that the lattice width of every polyhedron $\widehat Q_{i_t}$ in $\R^{n-d}$ is at most $\omega$, then $s^*(P)=+\infty$ also when we view $P$ as a polyhedron in $\R^{n-1}$ (Claim~\ref{cl:small-width});
\item
in step 1, a basis $\{b_{d+1},\dots,b_n\}$, a number $\varepsilon$, and a sequence $(Q_i)_{i\in\N}$ satisfying (a)--(c) do exist (Claim~\ref{cl:increase-dim});
\item
if we stop in step 3, then $F$ and $L$ are correctly determined (Claim~\ref{cl:no-split});
\item
if the condition of step 4 is true, then there exists $j\in\{0,1\}$ such that $P^j$ has infinite reverse split rank when viewed as a polytope in the affine space $\{x\in\R^n:ax=\beta+j\}$ (Claim~\ref{claim:faces}).
\end{itemize}

In the next four claims we prove the correctness of the procedure.
Recall that $\aff(P)=\R^d\times\{0\}^{n-d}$. Also, we identify the space $(\aff P)^\bot$ with $\R^{n-d}$; we will denote its variables by $x_{d+1},\dots,x_n$.
We denote by $e_j$ the unit vector in $\R^n$ with its only 1 in position $j$, for $j=1,\dots,n$.

\begin{myclaim}\label{cl:small-width}
If there exists an infinite subsequence of indices $i_1,i_2,\dots$ such that the lattice width of every polyhedron $\widehat Q_{i_t}$ in $\R^{n-d}$ is at most $\omega$, then $s^*(P)=+\infty$ also when we view $P$ as a polyhedron in $\R^{n-1}$.
\end{myclaim}

\begin{proof}
If the lattice width of every polyhedron $\widehat Q_{i_t}$ in $\R^{n-d}$ is at most $\omega$, then
for every $t\in\N$ there is a primitive direction $c_t\in\Z^{n-d}$ such that every polyhedron $\widehat Q_{i_t}$ has width at most $\omega$ with respect to $c_t$.
For each $t\in\N$, we can find a unimodular transformation $u_t$ that maps $c_t$ to $e_n$ and keeps the subspace $\aff P$ unchanged. The resulting polyhedra $u_1(Q_{i_1}),u_2(Q_{i_2}),\dots$ are still relaxations of $P$, and they have the same split rank as $Q_{i_1},Q_{i_2},\dots$, respectively (see Sect.~\ref{sec:unimod}).
Thus $\sup\{s(u_1(Q_{i_1})),s(u_2(Q_{i_2})),\dots\}=+\infty$.
By Lemma~\ref{lem:ub-chv}, there is an integer $N$ such that $N$ iterations of the CG closure operator are sufficient to reduce each $u_t(Q_{i_t})$ to a polyhedron contained in $\{x\in\R^n:x_n=0\}$. Then $s^*(P)=+\infty$ also when we view $P$ as a polyhedron in $\R^{n-1}$.
\end{proof}

Under the hypothesis of the above claim, by induction there are $F$ and $L$ satisfying the conditions of the theorem when $P$ is viewed as a polytope in $\R^{n-1}$. It is immediate to check that with this choice of $F$ and $L$ the conditions of theorem are also satisfied when $P$ is viewed as a polytope in $\R^n$.

From now on we can assume that the hypothesis of the previous lemma is not satisfied. Wlog, we assume that every polyhedron in the sequence $(\widehat Q_i)_{i\in\N}$ has (minimum) lattice width larger than $\omega$.

\begin{myclaim}\label{cl:increase-dim}
For every $\bar x\in\relint P$, there exist a rational basis $\{b_{d+1},\dots,b_n\}$ of $(\aff P)^\bot$, a rational number $\varepsilon>0$, and a sequence of rational polyhedra $(Q_i)_{i\in\N}$ such that:
\begin{enumerate}[\upshape(a)]
\item
$P(\bar x,\varepsilon)$ has the same integer points as $P$;
\item
$Q_i$ is a relaxation of $P(\bar x,\varepsilon)$ (and thus of $P$) for every $i\in\N$;
\item
$\sup_i s(Q_i)=+\infty$.
\end{enumerate}
\end{myclaim}

\begin{proof}
Since every $\widehat Q_i$ has lattice width larger than $\omega$, because of Lemma~\ref{lem:flat} every $\widehat Q_i$ contains a nonzero integer point $\hat y_i$. Since the origin belongs to $\widehat Q_i$, we can assume wlog that $\hat y_i$ is a primitive vector in $\Z^{n-d}$. For every $i\in\N$, there exists a unimodular linear transformation of $\R^{n-d}$ that maps $\hat y_i$ to $e_{d+1}$. Furthermore, each of these transformations can be extended to a unimodular linear transformation of $\R^n$ that maps $\aff P$ to itself. To simplify notation, we assume that every $Q_i$ coincides with its image via the latter transformation. Then every polyhedron in the sequence $(Q_i)_{i\in\N}$ is a relaxation of $P$ that contains a point $y_i$ of the form $y_i=x_i+e_{d+1}$ for some $x_i\in\aff P=\R^d\times\{0\}^{n-d}$.

For every $i\in\N$, define $z_i=y_i-\bar x$. Let $\bar z_i$ be the vector obtained from $z_i$ by rounding each entry to the closest integer.
Note that $\bar z_i\in\Z^{d+1}\times\{0\}^{n-d-1}$ and has its $(d+1)$-th component equal to 1. Then $\bar z_i$ is a primitive vector and therefore there exists a unimodular linear transformation $u_i$ such that $u_i(x)=x$ for every $x\in\aff P$ and $u_i(\bar z_i)=e_{d+1}$.
We then have
\[u_i(y_i)=u_i(\bar x+z_i)=u_i(\bar x)+u_i(z_i-\bar z_i)+u_i(\bar z_i)=\bar x+(z_i-\bar z_i)+e_{d+1},\]
where the equality $u_i(z_i-\bar z_i)=z_i-\bar z_i$ holds because $z_i-\bar z_i\in\aff P$ and $u_i(x)=x$ for every $x\in\aff P$.
Since $z_i-\bar z_i\in\aff P=\R^d\times\{0\}^{n-d}$ and has its components in the interval $[-1/2,1/2]$,
we obtain that every $u_i(Q_i)$ is a relaxation of $P$ that contains a point of the type $p_i+e_{d+1}$ for some $p_i\in\aff P$ such that $\|p_i-\bar x\|_{\infty}\le 1/2$ (see Fig.~\ref{fig:common-point}). Since $\bar x\in\relint P$, there exists $r>0$ such that $B:=B(\bar x,r)\cap\aff P\subseteq P$. Then $\conv(B,p_i+e_{d+1})\subseteq u_i(Q_i)$ for $i\in\N$. This implies that there exists a point of the type $\bar x+\varepsilon e_{d+1}$ that belongs to $u_i(Q_i)$ for every $i\in\N$ (for some rational $\varepsilon>0$). By choosing $\varepsilon$ sufficiently small, the polyhedron $\widetilde P=\conv(P,\bar x+\varepsilon e_{d+1})$ will have the same integer points as $P$. Note that $\widetilde P$ is a polyhedron of dimension $d+1$ and every $\widetilde Q_i=u_i(Q_i)$ is a relaxation of $\widetilde P$. We take $b_{d+1}=e_{d+1}$.

The conclusion now follows by iterating the arguments used in this proof until a full-dimensional polytope is obtained.
Note that we cannot use the same $\bar x$, since it does not lie in $\relint\widetilde P$. However, we can slightly perturb it, e.g.\ by taking the middle point of $\bar x$ and $\bar x + \varepsilon e_{d+1}$. After the last iteration, we will determine a rational basis $\{b_{d+1},\dots,b_n\}$ of $(\aff P)^\bot$ and a rational number $\varepsilon>0$ as required.
\end{proof}

\begin{figure}
\begin{center}
\psset{unit=9mm}
\psset{linewidth=.5pt}
\begin{pspicture}(10.7,4)
\psline[linestyle=dashed](0,0.5)(8,0.5)
\psline(3,0.5)(5,0.5)
\put(4,0.5){\circle*{.1}}
\put(3.9,0.1){$\bar x$}
\put(8.2,0.45){$\aff P$}
\psline[linestyle=dashed](0,3)(8,3)
\put(8.2,2.95){$\aff P+e_{d+1}$}
\put(1.2,3){\circle*{.1}}
\put(3,3){\circle*{.1}}
\put(6,3){\circle*{.1}}
\put(7.1,3){\circle*{.1}}
\psline(3,0.5)(1.2,3)(5,0.5)
\psline(3,0.5)(3,3)(5,0.5)
\psline(3,0.5)(6,3)(5,0.5)
\psline(3,0.5)(7.1,3)(5,0.5)
\psline[linestyle=dashed](0.5,0.1)(0.5,3.7)
\psline[linestyle=dashed](7.5,0.1)(7.5,3.7)
\psline[linestyle=dashed](4,0.5)(4,.7)
\psline[linestyle=dashed](4,0.95)(4,3.7)
\psline{<->}(.5,3.5)(3.99,3.5)
\psline{<->}(7.5,3.5)(4.01,3.5)
\put(2,3.7){$\scriptstyle1/2$}
\put(5.5,3.7){$\scriptstyle1/2$}
\put(3.925,.75){$\star$}
\end{pspicture}
\caption{Illustration of the proof of Claim~\ref{cl:increase-dim}. Each point in the higher part of the figure is a point of the type $p_i+e_{d+1}$ for some $p_i\in\aff P$ such that $\|p_i-\bar x\|_\infty\le1/2$. The base of the pyramids is the ball $B(\bar x,r)$. The pyramids have a common point of the form $\bar x+\varepsilon e_{d+1}$ for some $\varepsilon>0$, e.g., the one marked with an asterisk.}
\label{fig:common-point}
\end{center}
\end{figure}
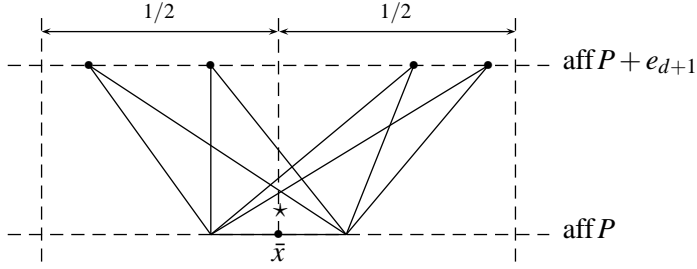

Notice that the sequence $(Q_i)_{i\in\N}$ consists of relaxations of $P(\bar x,\varepsilon')$ for every strictly positive rational number $\varepsilon'\le\varepsilon$.

The next claim shows that if we stop in step~3 at some iteration $k$, then $F$ and $L$ are correctly determined by the algorithm. The fact that $P(\bar x,\varepsilon)+L_k$ is lattice-free can be proved as in Sect.~\ref{sec:d}. Then the existence of a rational subspace $L$ containing $L_k$ such that $P(\bar x,\varepsilon)+L$ is lattice-free (which is required in step~3) follows from the discussion made in Sect.~\ref{sec:lf}.

\begin{myclaim}\label{cl:no-split}
Assume that, for every strictly positive rational number $\varepsilon'\le\varepsilon$, $P(\bar x,\varepsilon')+L_k$ is not contained in any split. Let $L$ be a rational subspace containing $L_k$ such that $P(\bar x,\varepsilon)+L$ is lattice-free. Then $P+L$ is relatively lattice-free and $\relint(P+L)$ is not contained in the interior of any split (i.e., the conditions of the theorem are satisfied with $F=P$).
\end{myclaim}

\begin{proof}
Assume by contradiction that there is a split $S$ such that $\relint(P+L)\subseteq\int S$. Then $\bar x\in\int S$. This implies that for $\varepsilon'>0$ small enough $P(\bar x,\varepsilon')+L_k\subseteq P(\bar x,\varepsilon')+L\subseteq S$, a contradiction.

We now prove that $P+L$ is relatively lattice-free.
Note that since $P$ is a face of $P(\bar x,\varepsilon)$, $P+L$ is a face of $P(\bar x,\varepsilon)+L$. Then either $\relint(P+L)$ is contained in the boundary of $P(\bar x,\varepsilon)+L$, or $\relint(P+L)\subseteq\int(P(\bar x,\varepsilon)+L)$.
The latter case immediately implies that $P+L$ is relatively lattice-free, as $P(\bar x,\varepsilon)+L$ is lattice-free.
So we assume that $\relint(P+L)$ is contained in the boundary of $P(\bar x,\varepsilon)+L$.

Let $H$ be a rational hyperplane containing $P+L$ and not containing any interior point of $P(\bar x,\varepsilon)+L$; note that $H$ is a supporting hyperplane for $P(\bar x,\varepsilon)+L$. We denote by $ax=\beta$ an equation defining $H$, where $a\in\Z^n$ is a primitive vector and $\beta\in\Z$. Assume wlog that $ax\ge\beta$ is a valid inequality for $P(\bar x,\varepsilon)+L$. Define the split $S=\{x\in\R^n:\beta\le ax\le\beta+1\}$. For $\varepsilon'>0$ sufficiently small, $P(\bar x,\varepsilon')+L_k\subseteq P(\bar x,\varepsilon')+L\subseteq S$, a contradiction.
\end{proof}

\begin{myclaim}\label{claim:faces}
In step 4, if there exists $M\in\R$ such that $Q_i\subseteq\{x\in\R^n:\beta-M\le ax\le \beta+M\}$ for every $i\in\N$, then there exists $j\in\{0,1\}$ such that $P^j$ has infinite reverse split rank when viewed as a polytope in the affine space $\{x\in\R^n:ax=\beta+j\}$.
\end{myclaim}

\begin{proof}
Denote by $H^0$ and $H^1$ the hyperplanes defining the split $S$. For $j\in\{0,1\}$, let $P^j(\bar x,\varepsilon)$ be the face of $P(\bar x,\varepsilon)$ induced by $H^j$.
The proof of the full-dimensional case shows that there exists $j\in\{0,1\}$ such that $P^j(\bar x,\varepsilon)$ has infinite reverse split rank when viewed as a polytope in the $(n-1)$-dimensional space $H^j$.
Since $P^j\subseteq P^j(\bar x,\varepsilon)$ and these two polytopes have the same integer points, $P^j$ has also infinite reverse split rank when viewed as a polytope in $H^j$.
\end{proof}

\subsection{An observation on $F$}

We conclude this section with an observation that gives some more information on the face $F$ in the statement of Theorem~\ref{th:main}.

\begin{observation}\label{obs:face}
Let $P\subseteq\R^n$ be an integral polytope with $s^*(P)=+\infty$, and let $F$ and $L$ be the output of the procedure of Sect.~\ref{sec:nec-outline} or Sect.~\ref{sec:non-full}. Then $F+L$ is a face of $P+L$.
\end{observation}

\begin{proof}
If the procedure terminates at the first iteration, then $F=P$ and the statement is trivial. Therefore we assume that the procedure ends at some iteration $k>1$. In this case $F$ and $L$ are determined by induction on a face $P^j$ of $P$, for some for some $j\in\{0,1\}$, where $P^j$ is viewed as a polytope in a rational hyperplane $H^j=\{x\in\R^n:ax=\beta+j\}$. Define $H^*=\lin H^j$.
Then, assuming the statement true by induction, $F+L$ is a face of $P^j+L$. Since $P^j\subseteq H^j$, and $L\subseteq H^*$, $P^j+L$ is a face of $P+L$. Then $F+L$ is a face of $P+L$.

\end{proof}

\section{Proof of necessity for unbounded polyhedra}\label{sec:nec-unbounded}

We prove here the necessity of conditions (i)--(ii) of Theorem~\ref{th:main} for unbounded polyhedra.
We assume that $P\subsetneq\R^n$, as for $P=\R^n$ we have $s^*(P)=0$ and there is nothing to prove.
We start with a lemma.

\begin{lemma}\label{lem:lin}
Let $P\subseteq\R^n$ be an integral polyhedron with $\lin P=\langle e_{k+1},\dots,e_n\rangle$ for some $k\in\{1,\dots,n\}$. Let $P'$ be the polyhedron $P\cap\langle e_1,\dots,e_k\rangle$ viewed as a convex set in the space $\langle e_1,\dots,e_k\rangle$ (which is equivalent to $\R^k$). Then $s^*(P)=s^*(P')$.
\end{lemma}

\begin{proof}
Let $\pi:\R^n\to\R^k$ be the map that drops the last $n-k$ components of every vector. Note that $\pi(P)=P'$. Furthermore, $\pi$ maps integer points to integer points.
Since every relaxation $Q$ of $P$ is such that $\lin Q\supseteq\langle e_{k+1},\dots,e_n\rangle$, $\pi$ induces a bijection between the relaxations of $P$ and those of $P'$. Also, $\pi$ induces a bijection between the splits of $\R^n$ whose lineality space contains $\langle e_{k+1},\dots,e_n\rangle$ and the splits of $\R^k$. We remark that if $S$ is a split of $\R^n$ whose lineality space contains $\langle e_{k+1},\dots,e_n\rangle$, then $\pi(\conv(Q\setminus\int S))=\conv(\pi(Q)\setminus\int(\pi(S)))$,
while if $\lin S$ does not contain $\langle e_{k+1},\dots,e_n\rangle$ then $\conv(Q \sm \int(S))=Q$ (i.e., $S$ has no effect when applied to a relaxation of $P$), as in this case $S$ does not contain any minimal face of $Q$. We conclude that if $Q$ is a relaxation of $P$ then $\pi(Q)$ is a relaxation of $P'$ with the same split rank. The lemma follows.
\end{proof}

Let $P\subsetneq\R^n$ be an integral polyhedron with $s^*(P)=+\infty$.
We now show that thanks to the above lemma we can reduce to the case $\lin P=\{0\}$. Indeed, if this is not the case, we can assume wlog that $\lin P=\langle e_{k+1},\dots,e_n\rangle$ for some $k\in\{1,\dots,n\}$. If we define $P'$ as in Lemma~\ref{lem:lin}, $P'$ is an integral polyhedron satisfying $\lin P'=\{0\}$ and $s^*(P')=+\infty$. Given a face $F'$ of $P'$ and a nonzero rational subspace $L'\subseteq\R^k$ such that (i)--(ii) are satisfied for $P'$, we have that by setting  $F=F'\times\R^{n-k}$ and $L=L'\times\{0\}^{n-k}$ (which is not contained in $\lin P$), conditions (i)--(ii) of Theorem~\ref{th:main} are satisfied for $P$.

Therefore in the following we assume that $\lin P=\{0\}$ (but $\rec P\ne\{0\}$).
Note that in this case the condition ``$L\not\subseteq\lin P$'' of Theorem~\ref{th:main} simplifies to ``$L\ne\{0\}$''.

A second useful lemma is now stated.

\begin{lemma}\label{lem:rlf}
Let $Q\subseteq\R^n$ be a rational polyhedron, and define $\widetilde Q=Q+\langle\rec Q\rangle$.
Then $Q$ is relatively lattice-free if and only if $\widetilde Q$ is relatively lattice-free.
\end{lemma}

\begin{proof}
Since $\aff\widetilde Q=\aff Q$ and $Q\subseteq\widetilde Q$, if $\widetilde Q$ is relatively lattice-free then $Q$ is relatively lattice-free as well.

To show the reverse implication, assume that there is an integer point $\tilde x\in\relint\widetilde Q$.
Since $Q$ is a rational polyhedron, we can write $\rec Q=\cone\{r_1,\dots,r_k\}$, where $r_1,\dots,r_k$ are integer vectors.
Then $\langle\rec Q\rangle=\langle \pm r_1,\dots,\pm r_k\rangle$. This implies that we can write $\tilde x=x_0+\sum_{i=1}^k\lambda_ir_i$, where $x_0\in \relint Q$ and $\lambda_1,\dots,\lambda_k\in\R$.
Define $x=x_0+\sum_{i=1}^k(1+\lambda_i-\ceil{\lambda_i})r_i=\tilde x+\sum_{i=1}^k(1-\ceil{\lambda_i})r_i$.
We claim that $x$ is an integer point in $\relint Q$. The integrality of $x$ follows from the fact that $x$ is a translation of $\tilde x$ by an integer combination of the integer vectors $r_1,\dots,r_k$. Furthermore, $x\in\relint Q$ as $x_0\in\relint Q$ and $1+\lambda_i-\ceil{\lambda_i}\ge0$ for $i=1,\dots,k$. Therefore $x$ is an integer point in $\relint Q$, and thus $Q$ is not relatively lattice-free.
\end{proof}

Since $s^*(P)=+\infty$, there is a sequence $(Q_i)_{i\in\N}$ of relaxations of $P$ such that $\sup s(Q_i)=+\infty$.
Define $L_0=\langle\rec P\rangle$. Note that $L_0\ne\{0\}$. Wlog, $L_0=\langle e_{k+1},\dots, e_n\rangle$ for some $k\in\{1,\dots,n\}$.
Also, define $\widetilde P=P+L_0$ and $\widetilde Q_i=Q_i+L_0$ for $i\in\N$.

Since the reverse split rank of $P$ is infinite, the same is true for its reverse CG rank; by Theorem~\ref{th:CG}, this implies that $P$ is relatively lattice-free. Then, by Lemma~\ref{lem:rlf}, $\widetilde P$ is also relatively lattice-free.
If $\relint\widetilde P$ is not contained in the interior of any split, then (i)--(ii) hold with $F=P$ and $L=L_0$.
Therefore in the remainder of the proof we assume that $\relint \widetilde P$ is contained in the interior of some split $S$.

\begin{claim}
$\widetilde Q_i$ is a relaxation of $\widetilde P$ for every $i\in\N$.
\end{claim}

\begin{proof}
Fix $i\in\N$ and assume that $\widetilde Q_i$ contains some integer point $\tilde x$; we prove that $\tilde x\in \widetilde P$. By using arguments that are similar to those in the proof of Lemma~\ref{lem:rlf}, $Q_i$ contains an integer point $x$ of the form $x=\tilde x+r$, where $r$ is an integer vector in $L_0$. Since $Q_i$ is a relaxation of $P$, we have $x\in P$. But then $\tilde x=x-r$ is in $\widetilde P$.
\end{proof}

Assume that $s^*(\widetilde P)<+\infty$, say $s^*(\widetilde P)=t$. Then, for every $i\in\N$, applying $t$ times the split closure operator to $\widetilde Q_i$ yields $\widetilde P$. If the same splits are applied to $Q_i$, we obtain a relaxation of $P$ which is contained in $\widetilde P$, which in turn is contained in $S$.
In other words, $t$ rounds of the split closure operator are sufficient to make $Q_i$ contained in $S$ for every $i\in\N$. As in the proof for polytopes (Sect.~\ref{sec:c}), this implies that at least one of the two faces of $P$ induced by the boundary of $S$ ($P^0$, say) has infinite reverse split rank. By induction, there exist a face $F$ of $P^0$ and a nonzero rational subspace $L$ satisfying (i)--(ii). The same choice of $F$ and $L$ is also good for $P$.

Therefore we now assume that $s^*(\widetilde P)=+\infty$. Since $\lin\widetilde P\ne\{0\}$, we can replicate the argument in the discussion following Lemma~\ref{lem:lin} and conclude that there exist a face $\widetilde F$ of $\widetilde P$ and a rational subspace $\widetilde L\not\subseteq\lin \widetilde P$ such that (i)--(ii) are fulfilled for $\widetilde P$. Let $H$ be any supporting hyperplane for $\widetilde F$. We now verify that the face $F$ of $P$ supported by $H$ and the space $L=\widetilde L$ satisfy the conditions for $P$.
To show that condition~(i) is satisfied, observe that a split contains $\relint(F+L)$ in its interior if and only if it contains $\relint(\widetilde F+L)$ in its interior, as $\widetilde F=F+\langle\rec F\rangle$; to check condition~(ii), one can use Lemma~\ref{lem:rlf}.
This concludes the proof of Theorem~\ref{th:main} for unbounded polyhedra.

\begin{remark}\label{rem:face}
Using Observation~\ref{obs:face}, one verifies that if $F$ and $L$ are obtained as above then $F+L$ is a face of $P+L$.
\end{remark}

\section{Connection with the mixed-integer case}\label{sec:mixed}

In this section we discuss a link between the concept of infinite reverse split rank in the pure integer case and that of infinite split rank in the mixed-integer case.

Fix $k\in\{1,\dots,n\}$, and consider $x_1,\dots,x_k$ as integer variables and $x_{k+1},\dots,x_n$ as continuous variables.
A split $S\subseteq\R^n$ is now defined as a set of the form $S=\{x\in\R^n:\beta\le ax\le\beta+1\}$ for some primitive vector $a\in\Z^k\times\{0\}^{n-k}$ and some integer number $\beta$. Note that every set of this type is also a split in the pure integer sense. The split closure of $Q$ is defined as in the pure integer case. The split rank of $Q$ is the minimum integer $k$ such that the $k$-th split closure of $Q$ coincides with $Q_I=\conv(Q\cap(\Z^k\times\R^{n-k}))$. Unlike the pure integer case, in the mixed-integer case such a number $k$ does not always exist; in other words, there are rational polyhedra with infinite split rank, see e.g.~\cite{CKS}. Note in fact that the example given in \cite{CKS} is obtained from the polytope presented in Sect.~\ref{sec:intro} by considering $x_3$ as the unique continuous variable and ``enlarging'' it along $x_3$. (We will develop this idea below.)
We remark, however, that the split closure of a rational polyhedron $Q$ asymptotically converges to $Q_I$ (with respect to the Hausdorff distance), as shown in \cite{delpia-weis}.

Given a rational polyhedron $Q$ and a valid inequality $cx\le\delta$ for its mixed-integer hull $Q_I$, we say that the split rank of $cx\le\delta$ is $k$ if the inequality is valid for the $k$-th split closure of $Q$ but not for the $(k-1)$-th split closure of $Q$.
The following theorem, which was proven in~\cite{delpia} and extends results presented in~\cite{basu-corn-margot}, characterizes the valid inequalities for $Q_I$ that have infinite split rank.

\begin{theorem} \label{th:rank}
Let $Q\subseteq\R^n$ be a rational polyhedron. For some fixed $k\in\{1,\dots,n\}$, define $Q_I=\conv(Q\cap(\Z^k\times\R^{n-k}))$ and let $\pi$ denote the orthogonal projection onto the space $\ang{e_1,\dots,e_k}$.
Let $cx \le \delta$ be a valid inequality for $Q_I$.
Then $cx \le \delta$ has infinite split rank for $Q$ if and only if
there exists a face $M$ of $\pi(\{x \in Q_I : cx = \delta\})$ such that $M \cap \pi(\{x \in Q : cx > \delta\})\neq \varnothing$ and $\relint M$ is not contained in the interior of any split.
\end{theorem}

The above result, compared with Theorem~\ref{th:main}, suggests that there is a connection between the integral polyhedra with infinite reverse split rank and the rational polyhedra with infinite split rank in the mixed-integer case. We propose such a connection below.

\begin{proposition}\label{prop:mixed}
Let $P\subseteq\R^n$ be an integral polyhedron with $s^*(P)=+\infty$. Let $F$ and $L$ be as in Theorem~\ref{th:main}, where we assume wlog $L=\ang{e_{k+1},\dots,e_n}$ for some $k\in\{1,\dots,n\}$.
Denote by $\pi$ the orthogonal projection onto the space $\ang{e_1,\dots,e_k}$, and define $\widetilde P=\pi(P)$.
Choose $\bar x\in\relint F$ and define $\tilde x=\pi(\bar x)$.
Then the rational polyhedron $Q=\clconv(\widetilde P,\tilde x + e_{k+1},\dots,\tilde x + e_n)$ has infinite split rank, where variables $x_1,\dots,x_k$ are integer and variables $x_{k+1},\dots,x_n$ are continuous.
\end{proposition}

\begin{proof}
Note that $\widetilde P$ is an integral polyhedron. We claim that $Q_I=\widetilde P$, where $Q_I=\conv(Q\cap(\Z^k\times\R^{n-k}))$.
If $x$ is an integer point in $\widetilde P$, then clearly $x\in Q_I$; since $\widetilde P$ is an integral polyhedron, this implies that $\widetilde P\subseteq Q_I$.
Assume by contradiction that $\widetilde P\subsetneq Q_I$. Then $Q$ contains a point $w \in (\Z^k\times\R^{n-k})\sm \widetilde P$.
Note that $z:=w + \sum_{i=k+1}^n \lambda_i e_i \in (P+L) \cap \Z^n$ for some $\lambda_{k+1},\dots,\lambda_n\in\R$. We can then proceed as in the proof of Claim~\ref{claim:relax} to obtain a contradiction.
%Assume by contradiction that $\widetilde P\subsetneq Q_I$. Then $Q$ contains a point $w \in (\Z^k\times\R^{n-k})\sm \widetilde P$. Define $\widetilde F=\pi(F)$.
%As $\tilde x \in \relint(\widetilde F)$, we deduce $w=\tilde y + v$, where $\tilde y \in \relint{\widetilde P}\cup \relint(\widetilde F)$, and $v \in \{0\}^{k} \times \R^{n-k}$ with all components between $0$ and $1$. Let $\bar y \in P$ be such that $\pi(\bar y)=\tilde y$, and note that we can assume that $\bar y\in \relint P \cup \relint F$.  Then there exist $\lambda_{k+1},\dots,\lambda_n\in\R$ such that $\bar y+\sum_{i=k+1}^n\lambda_i e_i$ is an integer vector, contradicting with the assumptions that $F+L$ and $P+L$ are relatively lattice-free. Therefore we conclude that $Q_I=\widetilde P$.

%Assume by contradiction that $\widetilde P\subsetneq Q_I$. Then $Q$ contains a point in $(\Z^k\times\R^{n-k})\sm \widetilde P$.
%This implies that there exists an index $i\in\{k+1,\dots,n\}$ such that $\tilde x+ e_i$ is in $\Z^k\times\R^{n-k}$.
%Then there exist $\lambda_{k+1},\dots,\lambda_n\in\R$ such that $\bar x+\sum_{j=k+1}^n\lambda_j e_j$ is an integer vector, contradicting with the assumptions that $\bar x\in\relint F$ and $F+L$ is relatively lattice-free. Therefore we conclude that $Q_I=\widetilde P$.

Define $\widetilde F=\pi(F)$ and let $\overline F$ be the minimal face of $\widetilde P$ containing $\widetilde F$.
Let $cx\le\delta$ be an inequality defining face $\overline F$ of $\widetilde P$, where wlog $c\notin L^\bot$.
Using Theorem~\ref{th:rank}, we show below that the inequality $cx\le\delta$ has infinite split rank for $Q$, thus implying that $Q$ has infinite split rank.

Note that $\pi(\{x \in Q_I : cx = \delta\})=\overline F$; we choose $M$ to be this set. Since $c\notin L^\bot$, the set $M \cap \pi(\{x \in Q : cx > \delta\})$ contains $\tilde x$ and thus it is nonempty. In order to apply Theorem~\ref{th:rank}, it remains to show that $\relint M$ (i.e., $\relint \overline F$) is not contained in the interior of any split. Assume by contradiction that there is a split $S$ such that $\relint \overline F\subseteq \int S$. Since $\overline F\subseteq L^\bot$, we can assume that $L\subseteq \lin S$.
Since $\overline F$ and $\widetilde F$ have the same dimension, and $\widetilde F\subseteq\overline F$, we have that $\relint(\widetilde F)\subseteq\int S$.
Then $\relint(F+L)\subseteq \int S$, a contradiction to condition~(i) of Theorem~\ref{th:main}.

\end{proof}

One might wonder why in Proposition~\ref{prop:mixed} the polyhedron $Q$ is not defined simply as $\clconv(P,\bar x + e_{k+1},\dots,\bar x + e_n)$, or perhaps $\clconv(P,\bar x + \lambda e_{k+1},\dots,\bar x +\lambda e_n)$ for some $\lambda>0$.
In fact, with this definition $Q$ might have finite split rank. For instance, let $P\subseteq\R^3$ be defined as the convex hull of the points
\[(0,0,0),\:(2,0,0),\: (0,2,0),\: (1,1,1),\: (1,1,-1).\]
$P$ is an integral polyhedron with infinite reverse split rank, as shown by the face $F=P$ and the linear space $L=\ang{e_3}$. Take $k=2$. We claim that if we choose any $\bar x\in\int P$ and any $\lambda>0$, then the polyhedron $Q=\clconv(P,\bar x+\lambda e_3)=\conv(P,\bar x+\lambda e_3)$ has finite split rank.
To see this, observe that $Q_I=P$ has five facets, defined by the following inequalities:
\begin{align*}
-x_1 \phantom{{}+x_2}+x_3&\le0\\
-x_1 \phantom{{}+x_2}-x_3&\le0\\
-x_2+x_3&\le0\\
-x_2-x_3&\le0\\
x_1+x_2\phantom{{}+x_3}&\le2.
\end{align*}
The last inequality is valid for $Q$, thus its split rank is zero. One verifies that for each of the first four inequalities there is no $M$ satisfying the conditions of Theorem~\ref{th:rank}; therefore all these inequalities have finite split rank. It follows that $Q$ has finite split rank.

We now present a result which is, in a sense, the inverse of Proposition~\ref{prop:mixed}.
In order to prove it, we will use of the following lemma, shown in \cite[Lemma~2.1]{delpia}.

\begin{lemma} \label{lem: relint}
Let $Q\subseteq\R^n$ be a rational polyhedron. For some fixed $k\in\{1,\dots,n\}$, let $\pi$ denote the orthogonal projection onto the space $\ang{e_1,\dots,e_k}$.
Let $cx \le \delta$ be an inequality, and let $M$ be a polyhedron contained in $\pi (\{x \in Q : cx \ge \delta\})$.
If $M \cap \pi (\{x \in Q : cx > \delta\}) \neq \varnothing$, then $\relint M \subseteq \pi (\{x \in Q : cx > \delta\})$.
\end{lemma}

\begin{proposition} \label{mi2}
Let $Q\subseteq\R^n$ be a rational polyhedron.
For some fixed $k\in\{1,\dots,n\}$, define $Q_I=\conv(Q\cap(\Z^k\times\R^{n-k}))$ and let $\pi$ denote the orthogonal projection onto the space $\ang{e_1,\dots,e_k}$.
Let $cx \le \delta$ be a valid inequality for $Q_I$ with infinite split rank for $Q$.
Then $\pi(\{x \in Q_I : cx = \delta\})$ has infinite reverse split rank in the space $\R^n$, where all variables are integer.
\end{proposition}

\begin{proof}
By Theorem \ref{th:rank}, there exists a face $M$ of $P := \pi(\{x \in Q_I : cx = \delta\})$ such that $M \cap \pi (\{x \in Q : cx > \delta\}) \neq \varnothing$ and $\relint M$ is not contained in the interior of any split (in the mixed-integer sense).
Let $L = \ang{e_{k+1},\dots,e_n}$ and note that $L\not\subseteq\lin P$; moreover, $\relint (M+L)$ is not contained in the interior of any split (in the pure integer sense).

Let $G$ be a face of $P$ that contains $M$.
Note that $G$ is contained in $\pi (\{x \in Q : cx \ge \delta\})$.
As $G \cap \pi (\{x \in Q : cx > \delta\}) \neq \varnothing$, it follows by Lemma~\ref{lem: relint} that $\relint G \subseteq \pi (\{x \in Q : cx > \delta\})$.
The set $\{x \in Q : cx > \delta\}$ contains no point with the first $k$ components integer, and thus its projection contains no integer point, implying that both $G$ and $G+L$ are relatively lattice-free.
Hence by Theorem~\ref{th:main} (with $F=M$), $s^*(P)=+\infty$.

\end{proof}

\section{Concluding remarks}\label{sec:concl}

\subsection{On the dimension of $L$}

As illustrated in the introduction, Theorem~\ref{th:main} has strong similarities with Theorem~\ref{th:CG}, which characterizes the integral polyhedra with infinite reverse CG rank. One of the differences between the two statements is that in Theorem~\ref{th:CG} the subspace $L$ has dimension one. We show below that $L$ cannot be assumed to have dimension one in Theorem~\ref{th:main}.

Consider the integral polytope $P$ in $\R^4$ defined by
\[P = \conv\{(0,0,0,0),(1,0,0,0), (0,1,0,0)\}.\]
Note that $P$ lives in the linear subspace $\R^2 \times \{0\}^2$.

First we show that $P$ has infinite reverse split rank.
In order to do so, by our main result, it is sufficient to give a nonzero rational linear subspace $L\subseteq\R^4$ such that
$P+L$ is relatively lattice-free and $\relint(P+L)$ is not contained in the interior of any split.

Let $L$ be the linear subspace of $\R^4$ generated by vectors
$$
v^1 = (1/2,0,1,0), \qquad v^2 = (0,1/2,0,1).
$$
Consider the polytope $P'$ obtained from $P$ by projecting out variables $x_3$ and $x_4$, i.e., $P' = \conv\{(0,0),(1,0),(0,1)\}$.
(See Fig.~\ref{fig:v}(a) for the drawings of $P'$ and the lattice $\Z^2$.)
Consider also the lattice $Y$ in $\R^2$ obtained as the projection of $\Z^4$ onto $\R^2 \times \{0\}^2$ by means of $L$.
More formally, a point $y \in \R^2$ is in $Y$ if and only if there exists $\ell \in L$ such that $(y,0,0)+\ell \in \Z^4$.
It can be checked that $Y$ is the lattice $\frac{1}{2}\Z^2$.
(See Fig.~\ref{fig:v}(b) for the drawing of $P'$ and the lattice $Y$.)

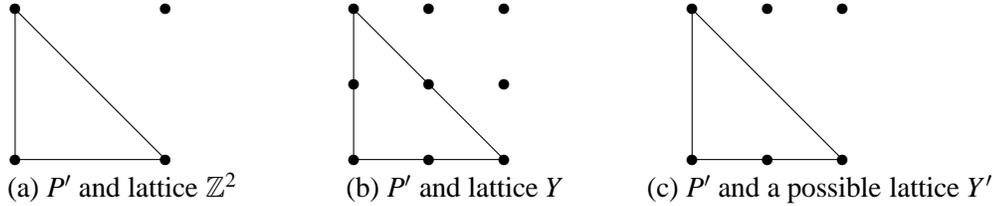
\begin{figure}[h]
\begin{center}
\begin{pspicture}(11,2)(0,-0.5)
\multiput(0,0)(0,2){2}{\multiput(0,0)(2,0){2}{\put(0,0){\circle*{.15}}}}
\pspolygon[linewidth=.4pt](0,0)(2,0)(0,2)
\put(-0.1,-0.5){(a) $P'$ and lattice $\Z^2$}
\multiput(4.5,0)(0,1){3}{\multiput(0,0)(1,0){3}{\put(0,0){\circle*{.15}}}}
\pspolygon[linewidth=.4pt](4.5,0)(6.5,0)(4.5,2)
\put(4.4,-0.5){(b) $P'$ and lattice $Y$}
\multiput(9,0)(0,2){2}{\multiput(0,0)(1,0){3}{\put(0,0){\circle*{.15}}}}
\pspolygon[linewidth=.4pt](9,0)(11,0)(9,2)
\put(8.4,-0.5){(c) $P'$ and a possible lattice $Y'$}
\end{pspicture}
\end{center}
\caption{Illustrations of $P'$ and different lattices $\Z^2$, $Y$, and $Y'$.}
\label{fig:v}
\end{figure}

We say that a set $Q$ is \emph{$Y$-free} if it contains no point of $Y$ in its interior, and a \emph{$Y$-split} is the convex hull of two parallel hyperplanes containing points in $Y$ that is $Y$-free.
As $P'$ is $Y$-free, one checks that $P+L$ is lattice-free.
Moreover, since $P'$ is not contained in the interior of any $Y$-split, one verifies that $P+L$ is not contained in any split.
Therefore $L$ satisfies the desired conditions and thus $P$ has infinite reverse split rank.

We now show that for every face $F$ of $P$, there is no nonzero rational vector $v \in \R^n$ such that conditions (i)--(ii) of Theorem~\ref{th:main} hold for $L=\ang v$.
We already observed in the introduction that $F$ must have dimension at least two, thus we only consider the case $F=P$.

Assume that $P+\ang v$ is relatively lattice-free.
If $v_3=v_4=0$, it is easy to check that $\relint(P+\ang v)$ is always contained in the interior of a split.
Therefore assume now that $(v_3,v_4)\neq (0,0)$ and, by scaling, that $v_3$ and $v_4$ are coprime integers.
Consider the lattice $Y'$ in $\R^2$ obtained as the projection of $\Z^4$ onto $\R^2 \times \{0\}^2$ by means of $v$.
More formally, a point $y \in \R^2$ is in $Y'$ if and only if there exists $\lambda \in \R$ such that $(y,0,0)+ \lambda v \in \Z^4$.
$Y'$ is the lattice generated by the vectors $(1,0)$, $(0,1)$, and $(v_1,v_2)$.
(See Fig.~\ref{fig:v}(c) for a drawing of $P'$ and a possible lattice $Y'$.)
Note that the lattice $Y'$ can contain at most one of the three points $(1/2,0)$, $(0,1/2)$, and $(1/2,1/2)$ (and in particular $Y'$ is different from the lattice $Y$).
Since $P+\ang v$ is relatively lattice-free, the polytope $P'$ is $Y'$-free.
Hence $P'$ is a $Y'$-free triangle with vertices in $Y'$ and at most one of the three middle points of its edges is in $Y'$.
This is well known to imply that $P'$ is contained in the interior of a $Y'$-split, which in turn shows that $\relint(P+\ang v)$ is contained in the interior of a split.

\begin{remark}
The previous example also shows that there exist 0/1 polytopes with infinite reverse split rank. (The example can be made full-dimensional, if one is interested in this further condition.) This contrasts with the fact that the split rank (and even the CG rank) of 0/1 polytopes in dimension $n$ is bounded by a function of $n$ (see \cite{balas,Eisc}).
\end{remark}

%\subsection{The reverse split rank of 0/1-polytopes}

%We give an example showing that the reverse split rank can be infinite already for 0/1-polytopes.
%Let $P\subseteq\R^4$ be the convex hull of the points
%\[(0,0,0,0),\, (1,0,0,0),\, (1,1,0,0),\, (0,0,1,0),\, (0,0,1,1).\]
%In order to show that $s^*(P)=+\infty$, we give $F$ and $L$ as in the statement of Theorem~\ref{th:main}. We choose $F=P$ and $L=\ang{(1,-1,0,0),(0,0,1,-1)}$.
%Since $F=P$ and $P$ is full-dimensional, we need to show that $P+L$ is lattice-free and $\int(P+L)$ is not contained in the interior of any split.

%Let $\widetilde P$ be the projection of $P$ onto the space $L^\bot$. One verifies that $\widetilde P$ is the convex hull of the points
%\[(0,0,0,0),\, (1/2,1/2,0,0),\, (1,1,0,0),\, (0,0,1/2,1/2),\, (0,0,1,1).\]
%Furthermore, $\Z^4$ projects down to the lattice $Y$ generated by the vectors $(1/2,1/2,0,0)$ and $(0,0,1/2,1/2)$.
%One can check that $\relint(\widetilde P)$ contains no point of the projected lattice, thus $\widetilde P + L=P+L$ is lattice-free.
%Moreover, there is no $Y$-split containing $\widetilde P$, thus $\int(P+L)$ is not contained in the interior of any split.

\subsection{On the necessity of considering faces}

In order to determine whether a polyhedron has infinite reverse split rank, all faces need to be considered in Theorem~\ref{th:main}, while this is not the case for the reverse CG rank ($F=P$ is the only interesting face in that case). We now show that this ``complication'' is necessary.

Let $P\subseteq\R^4$ be defined as the convex hull of points $(0,0,0,0)$, $(1,0,0,0)$, $(1,2,0,0)$, and $(1,0,2,0)$.
If $F$ is the face of $P$ induced by equation $x_1=1$, and  $L=\langle e_4\rangle$, then the conditions of the theorem are satisfied; thus $s^*(P)=+\infty$. However, the conditions are not fulfilled if we choose $F=P$ and the same $L$, as $\relint(P+L)$ is contained in the interior of the split $\{x\in\R^4: 0 \leq x_1 \leq 1\}$. Indeed one can verify that there is no subspace $L'$ such that the conditions are satisfied with $F=P$.

\subsubsection*{Acknowledgements} This work was supported by the {\em Progetto di Eccellenza 2008--2009} of {\em Fondazione Cassa di Risparmio di Padova e Rovigo}. Yuri Faenza was supported by the German Research Foundation (DFG) within the Priority Programme 1307 Algorithm Engineering.
The authors are grateful to two anonymous referees, whose detailed comments helped us to improve the paper.

\end{document}